\documentclass[UKenglish,12pt]{article}
\usepackage[utf8]{inputenc}

\usepackage[margin=0.8in]{geometry}

\usepackage{amssymb}
\usepackage{amsthm}
\usepackage{amsmath}
\usepackage{amsfonts}
\usepackage{mathtools}
\usepackage{tikz-cd}
\usepackage{leftidx}
\usepackage{bbm}
\usepackage{enumitem}
\usepackage{amsthm}
\usepackage{thmtools,thm-restate}
\usepackage{xcolor}
\usetikzlibrary{matrix,arrows,decorations.pathmorphing}
\usepackage{stmaryrd}

\theoremstyle{plain}
\newtheorem{thm}{Theorem}[section] 
\newtheorem{lem}[thm]{Lemma}
\newtheorem{prop}[thm]{Proposition}
\newtheorem{cor}[thm]{Corollary}

\newtheorem{introtheorem}{Theorem}[section]

\newtheorem{introcor}{Corollary}[section]

\theoremstyle{definition}
\newtheorem{defn}[thm]{Definition} 
\newtheorem{exmp}[thm]{Example}
\newtheorem{rem}[thm]{Remark}
\newtheorem{notation}[thm]{Notation}

\newcommand{\be}{\begin{enumerate}[label=(\alph*) ,leftmargin=*]}
\newcommand{\ben}{\begin{enumerate}[label=(\arabic*), leftmargin=*]}
\newcommand{\ee}{ \end{enumerate} }

\newcommand{\Mon}{\mathsf{Mon}}

\newcommand\und{\underline}
\newcommand{\C}{\mathcal{C}}
\newcommand{\U}{\mathcal{U}}
\newcommand{\V}{\mathcal{V}}
\newcommand{\A}{\mathcal{A}}
\newcommand{\D}{\mathcal{D}}
\newcommand{\PP}{\mathcal{P}}
\newcommand{\X}{\mathcal{X}}
\newcommand{\T}{\mathcal{T}}
\newcommand{\I}{\mathcal{I}}

\newcommand{\EE}{\mathbb{E}}
\newcommand{\bE}{\mathbb{E}}
\newcommand{\fs}{\mathfrak{s}}
\newcommand{\II}{\mathbb{I}}
\newcommand{\PPP}{\mathbb{P}}
\newcommand{\DD}{\mathbb{D}}
\newcommand{\CD}{\und{\C}_\D}
\newcommand{\SD}{\Sigma_\D}

\newcommand{\SR}{\mathcal{S}(\R)}

\newcommand{\bsm}{\left[ \begin{smallmatrix}}
\newcommand{\esm}{\end{smallmatrix} \right]}

\renewcommand{\S}{\mathcal{S}}

\newcommand{\R}{\mathcal{R}}

\newcommand\blfootnote[1]{%
  \begingroup
  \renewcommand\thefootnote{}\footnote{#1}%
  \addtocounter{footnote}{-1}%
  \endgroup
}

\title{Right triangulated categories: As extriangulated categories, aisles and co-aisles}
\author{Aran Tattar}
\date{\today}
\begin{document}

\maketitle
\abstract{Right triangulated categories can be thought of as triangulated categories whose shift functor is not an equivalence. We give intrinsic characterisations of when such categories have a natural extriangulated structure and are appearing as the (co-)aisle of a (co-)t-structure in an associated triangulated category.}

\section{Introduction}
\blfootnote{ \noindent 2020 Mathematics Subject Classification: 18G50,  18G80, 18E30, 18G25. \\ \: Keywords: Right triangulated, extriangulated, t-structure, co-t-structure. 
\\  Acknowledgments: The author would like to thank their PhD supervisor Sibylle Schroll for many helpful discussions.}
Extriangulated categories were introduced by Nakaoka and Palu in \cite{nakaoka2019extriangulated} as a simultaneous generalisation of exact categories and triangulated categories. The framework of extriangulated categories allows one to axiomatise properties of categories that have structural similarities to exact and/or triangulated categories but fall into neither class, for  example extension closed subcategories of triangulated categories. The formalism of extriangulated categories then allows homological algebra to be applied to such categories which has been done successfully by many authors, for example \cite{enomoto2021classifying, padrol2019associahedra,  Zhou2018}.

The data of a right triangulated category (or suspended category), first introduced in \cite{Keller1987} consists of an additive category $\R$, an endofunctor $\Sigma:\R \to \R$ called `the shift of $\R$' and a class of right triangles of the form $A \to B \to C \to \Sigma A$ subject to essentially the same axioms as the triangles of a triangulated category (see Definition \ref{def:righttri}). Informally, a right triangulated category is a triangulated category whose shift functor is not necessarily an equivalence. Such categories (or their left-handed analogues) have been the subject of study in many articles \cite{Assem2008, Beligiannis1992, Keller1990,  li2015triangulation, Lin2007}. 
 In \cite{assem1998right}, the class of `right triangulated categories with right semi-equivalence' were introduced, these are the right triangulated categories whose shift functor is fully faithful and with image that is closed under extensions. Such right triangulated categories enjoy  homological properties close to those of triangulated categories, we formalise this similarity by showing the following. 

\begin{introtheorem}  \label{inttheorem:righttri1}(Corollary \ref{cor:righttriextri}.)
A  right triangulated category has the natural structure of an extriangulated category if and only if the shift functor is a right semi-equivalence.  
\end{introtheorem}

Moreover, we are able to characterise which extriangulated categories have a natural right triangulated structure. 

\begin{introtheorem} \label{intthm: righttri2} (Theorem \ref{thm:nicerighttriextri}.)
The extriangles of an extriangulated category  $(\mathcal{C}, \EE, \fs)$ induce a right triangulated structure with right semi-equivalence on $\C$ if and only if  the morphism $X \to 0$ is an $\EE$-inflation for all objects $X \in \C$. In other words,  the zero object of $\C$ is the only $\EE$-injective object and there are enough $\EE$-injectives.
\end{introtheorem}

To prove this, we use an extriangulated generalisation of the constructions of a right triangulated quotient category from a contravariantly finite subcategory  (of an additive category) due to \cite{beligiannis1994left} and \cite{assem1998right} (see Proposition \ref{thm:abm}). We also use ideas from relative homological algebra  to characterise which extriangulated structures give rise to right triangulated quotient categories (with right semi-equivalence) (see Proposition \ref{funcyisom}).  

The construction of these right triangulated structures  can be thought of as a `one-sided' analogue of the triangulated structure of the stable category of a Frobenius exact (or even extriangulated) category \cite{Heller1960, happel1988triangulated}. See \cite{Beligiannis2000a,Iyama2008}  for similar constructions. Unsurprisingly, given the terminology, this one-sidedness also appears in Theorem \ref{intthm: righttri2} as triangulated categories are precisely the extriangulated categories satisfying both this condition and its dual \cite[Proposition 3.22]{nakaoka2019extriangulated}.

Standard examples of right triangulated categories with right semi-equivalence are aisles of t-structures  and co-aisles of co-t-structures  in triangulated categories introduced by \cite{Beuilinson1982} and \cite{Bondarko2010, Pauksztello2008} respectively (we recall the definitions in  Definition \ref{defn:torsion}). 


 We aim to characterise precisely which right triangulated categories are appearing as aisles and co-aisles. To every right triangulated category $\R$ there is an associated triangulated category: the stabilisation $\SR$ (see Section \ref{section:costab} for details and construction). In the case where $\R$ has a right semi-equivalence, $\SR$ can be thought of as the smallest triangulated category containing $\R$ as a subcategory. Furthermore, we show in Lemma \ref{lem:aisleinstab} that if $\R$ is a (co-)aisle in a triangulated category then it is also a (co-)aisle in $\SR$. Thus we look to characterise when $\R$ is a (co-)aisle in $\SR$.

We give an intrinsic characterisations of when a right triangulated category with right semi-equivalence, $\R$, appears as the co-aisle of a co-t-structure in terms of internal torsion pairs of $\R$ and homological properties.  

\begin{introtheorem} \label{intthm:coaisle} (Theorem \ref{thm:coaisles}.) Let $\R$ be a right triangulated category with right semi-equivalence. Then the following are equivalent
\be 
\item  $\R$ is the co-aisle of a co-t-structure $(\U, \R)$  in $\S(\R)$;
\item $\R$ has enough projectives;
\item There is a torsion pair $(\mathsf{Proj}_{\EE} \R, \Sigma \R)$ in $\R$.
\ee
Moreover, suppose that for all $A, B \in \R$ we have that $\R(A, \Sigma^i B ) =0 $ for $i>>0$, that is, $\R$ is \textit{bounded}. Then the above conditions are also equivalent to \be \item[(d)] $\mathsf{Proj}_{\EE}\R$ is a silting subcategory of $\S(\R)$. \ee
\end{introtheorem} 

Part (d) of the above result adds to the interpretations of silting subcategories in a triangulated category, which are surveyed in \cite{AngeleriHuegel2019a}. We also note that this result can be thought of as a generalisation of \cite[Theorem 4.2]{Assem2008} and that the work of \cite{MendozaHernandez2013} is useful in the proof. We also obtain the following as a direct consequence.

\begin{introcor} (Corollary \ref{cor:siltingcorrespondence}.)
There is a correspondence between silting subcategories of triangulated categories and bounded right triangulated categories with right semi-equivalence that have enough projectives.
\end{introcor}

For the case of t-structures, there are related works \cite{AlonsoTarrio2000, Keller1988, Laking2020} that give various characterisations of aisles. We note that our approach differs in the sense that we look to give characterisations intrinsic to the right triangulated category, that is, the aisle, rather than properties of the aisle related to the ambient triangulated category. 

\begin{introtheorem} \label{intthm:aisle} (Theorem \ref{thm:aisles}.) Let $\R$ be a right triangulated category with right semi-equivalence.
Then the following are equivalent
\be 
\item $\R$ is the aisle of a t-structure $(\R, \V)$  in $\S(\R)$;
\item There is a torsion pair $(\Sigma \R, \mathcal{F})$ in $\R$;
\item $\SR$ is equivalent to the co-stabilisation of $\R$. That is, $\SR$ has the following universal property: Every right triangle functor $G:\T \to \R$ with $\T$ being a triangulated category factors uniquely through $\SR$  (see Section \ref{section:costab}).
\ee
\end{introtheorem}

An immediate application is that the torsion pairs in Theorems \ref{intthm:coaisle}  and \ref{intthm:aisle} allow us to describe intermediate (co-)t-structures (Proposition \ref{prop:intermediate}).
We also show that in the case of Frobenius extriangulated categories, aisles of t-structures in the triangulated stable category may be constructed as (shifts of) right triangulated quotients.  
\begin{introtheorem}  \label{intthm:aislequotient} (Theorem \ref{thm:quotientshiftaisle}.) Let $\C = (\C, \EE, \fs)$ be a Frobenius extriangulated  category and $(\U,\V)$ be a t-structure in the triangulated stable category $\underline{\C}$ and set $\D := \Sigma^{-1}\V$. Then there is an equivalence of right triangulated categories $\U \cong \Sigma_\D \underline{\C}_\D$, where  $\underline{\C}_\D$ denotes the stable category of $\C$ by the ideal of morphisms factoring through objects of $\D$  and $\Sigma_\D$ is its shift functor.
\end{introtheorem}

In  \cite[Proposition 3.9]{saorin2011exact}, it was shown that t-structures in an algebraic triangulated category correspond bijectively to certain complete cotorsion pairs in the associated Frobenius exact category. Additionally, it was observed in \cite[Proposition 2.6]{Nakaoka2011} that t-structures in a triangulated category are precisely cotorsion pairs satisfying a shift closure property.  Thus these results and Theorem \ref{intthm:aislequotient} complement each other since, via these bijections, a t-structure $(\U, \V)$ corresponds to a cotorsion pair $(\U, \Sigma^{-1} \V)$ (Remark~\ref{rem:SS,Ncomplementary}).

The article is organised as follows. In Section \ref{sec:background} we recall the necessary background material on relative homological algebra, extriangulated categories and right triangulated categories. In Section \ref{sec:rightriextri} we show how one may use relative homological algebra to construct new extriangulated structures and characterise the projectives and injectives of these new structures. We then investigate how such extriangulated structures induce a right triangulated structure on a quotient category and use this to prove Theorems \ref{inttheorem:righttri1} and \ref{intthm: righttri2}. 
We begin Section \ref{sec:aislescoaisles} with a discussion of torsion pairs in a right triangulated category with right semi-equivalence. We then go on to prove the characterisations of right triangulated categories as (co-)aisles of (co-)t-structures of Theorems \ref{intthm:coaisle} and \ref{intthm:aisle} and use these to describe related classes of (co-)t-structures. 
We end in Section \ref{sec:aislesquotients} by proving Theorem \ref{intthm:aislequotient} using the right triangulated quotient categories from Section \ref{sec:rightriextri}.

We note that, throughout, dual results for left triangulated categories (with left semi-equivalence) hold but remain unstated. 
Categories are assumed to be additive and idempotent complete. When we say an additive subcategory, we mean a subcategory that is closed under isomorphisms, direct sums and direct summands. 

\section{Background} \label{sec:background}

\subsection{Relative homological algebra}

We briefly recall some basic definitions from relative homological algebra introduced in \cite{Hochschild1956} and state some useful properties. 
Let $\mathcal{C}$ be an additive category and let $\mathsf{Mor}_\mathcal{C}$ denote the category of morphisms in $\mathcal{C}$.

\begin{defn}
Let $\mathcal{D} \subset \mathsf{Ob}(\mathcal{C})$ be a class of objects. A morphism $f: A \to B$ in $\mathcal{C}$ is \textit{$\mathcal{D}$-monic} all morphisms $A \to D$ with $D \in \mathcal{D}$ factor through $f$
\[ \begin{tikzcd} & A \arrow[r, "f"] \arrow[d, "\forall"'] & B \arrow[dl, "\exists", dashed] \\
{} \arrow[r, phantom, " \mathcal{D} \ni" marking, near end] &D.  & {} 
\end{tikzcd} \]
or, equivalently, $\C(f, -)|_\D : \C(B,-)|_\D \to \C(A,-)|_\D$ is an epimorphism. By $\mathsf{Mon}(\mathcal{D})$ we denote the class of all $\mathcal{D}$-monic morphisms in $\mathcal{C}$. We define the notion of a $\mathcal{D}$-epic morphism and the class $\mathsf{Epi}(\mathcal{D})$ dually.  

Similarly, for a class of morphisms $\omega \subset \mathsf{Ob}(\mathsf{Mor}_\mathcal{C})$, an object $J \in \mathcal{C}$ is \textit{$\omega$-injective} if for every morphism $f: A \to B$ in $\omega$, all morphisms $A \to J$ factor through $f$ \[ \begin{tikzcd}[sep= large] A \arrow[r, "\forall f \in \omega"] \arrow[d, "\forall"'] & B \arrow[dl, "\exists", dashed] \\ J. & \end{tikzcd} \]
or, equivalently, $\C(f, J)$ is an epimorphism for all $f \in \omega$. By $\mathsf{Inj}(\omega)$ we denote the class of all $\omega$-injective objects in $\mathcal{C}$. We define the notion of an $\omega$-projective object and the class $\mathsf{Proj}(\omega)$ dually.
\end{defn}

\begin{exmp} 
\be 
\item $\mathcal{C}$-monics are precisely sections.
\item $\mathsf{Inj}{\mathcal{C}}$-monics are monomorphisms.
\item A left $\D$-approximation \cite{auslander1991applications} of an object $A$ is just a $\D$-monic morphism $A \to B$ such that $B \in \D$. 
\item In an extriangulated category (defined in the next section) $(\C, \EE, \fs)$ the $\EE$-injectives are precisely $\mathsf{Inj}(\{\EE\text{-inflations}\})$. 
\ee
\end{exmp}

\begin{defn}
A commutative square 
\[ \begin{tikzcd} A \ar[r, "x"] \ar[d, "f"'] & B \ar[d, "f'"] \\ A' \ar[r, "x'"'] & B' \end{tikzcd}\] in $\C$ is a \textit{weak pushout} if for all pairs of morphisms $g:A' \to C$, $y:B \to C$ such that $gf=yx$, there exists a (not necessarily unique) morphism $h:B' \to C$ such that $hf' = y$ and $hx'=g$:
\[ \begin{tikzcd} A \ar[r, "x"] \ar[d, "f"'] & B \ar[d, "f'"] \ar[ddr, bend left, "y"] & \\ A' \ar[r, "x'"'] \ar[drr, bend right, "g"'] & B' \ar[dr, dashed, "\exists h"] & \\ & & C. \end{tikzcd}\]
\end{defn}

We collect some useful properties.  

\begin{lem} \label{biglez} Let $\mathcal{D}$ be a class of objects in $\mathcal{C}$ and $\omega$ a class of morphisms. Then the following hold
\be 
\item $\mathsf{Mon}(\mathcal{D})$ is closed under composition and retracts;
\item $\mathsf{Mon}(\mathcal{D})$ is left divisive, that is, $gf \in \mathsf{Mon}(\mathcal{D})$ implies that $f \in \mathsf{Mon}(\mathcal{D})$; 
\item  $\mathsf{Mon}(\mathcal{D})$ is closed under weak pushouts. 
\ee
\end{lem}
\begin{proof}
(a) and (b) are easily verified. Let us show (c). Let \[ \begin{tikzcd} A \ar[r, "x"] \ar[d, "f"'] & B \ar[d, "f'"] \\ A' \ar[r, "x'"'] & B' \end{tikzcd}\] be a weak pushout square in $\C$ with $x \in \mathsf{Mon}\D$. We must show that $x' \in \mathsf{Mon}\D$. To this end, let $g: A' \to D$ be a morphism with $D \in \D$. Then, as $x$ is $\D$-monic, there exists a morphism $y: B \to D$ such that $yx = gf$. Now, by the weak pushout property there exists a morphism $h:B'\to C$ such that $hx'=g$ as required. 
\end{proof}

The interested reader may look at  \cite{maranda1964injective} and \cite{rump2009triadic} for more properties of $\D$-monics and $\omega$-injectives. For other applications of these notions, look, for instance in \cite{beligiannis2007homological, Enochs2000, hugel2015silting}.

\subsection{Extriangulated categories}
In this section we put the background on extriangulated categories following \cite[
Section 2]{nakaoka2019extriangulated} where such categories were introduced. Let $\C$ be an additive category  and $\EE: \C^{op} \times \C \to \mathsf{Ab}$ be an additive bifunctor.

\begin{defn} For any $A,C \in \C$, two pairs of composable morphisms in $\C$ 
\[ A\overset{x}{\longrightarrow}B\overset{y}{\longrightarrow}C\ \ \text{and}\ \ A\overset{x^{\prime}}{\longrightarrow}B^{\prime}\overset{y^{\prime}}{\longrightarrow}C \]
 are \textit{equivalent} if there exists an isomorphism $b: B \to B'$ such that 
 \[ \begin{tikzcd}[ampersand replacement = \&]  A \arrow[r, "x"] \ar[d, equal] \& B \arrow[d, "b", "\cong"'] \arrow[r, "y"] \& C \ar[d, equal] \\   A \arrow[r, "x'"'] \&  B' \arrow[r, "y'"'] \& C \end{tikzcd} \] commutes. 
 We denote the equivalence class of $A\overset{x}{\longrightarrow}B\overset{y}{\longrightarrow}C$ by $[A\overset{x}{\longrightarrow}B\overset{y}{\longrightarrow}C]$, and  by $\S(C,A)$ we denote the class of all such equivalence classes.
\end{defn}


\begin{notation}
 For $a:A \to A'$ write $a_{\ast} = \EE(-,a):\EE(-,A) \to \EE(-,A')$. Similarly, we write $c^{\ast} = \EE(c,-): \EE(C',-) \to \EE(C,-)$ for $c:C \to C'$.     
\end{notation}

\begin{defn} \cite[Definitions 2.9, 2.10]{nakaoka2019extriangulated}
An assignment $\fs_{C, A}: \EE(C, A) \to \mathcal{S}(C, A)$ for all $C, A \in \C$ is an \textit{additive realisation of $\EE$} if it satisfies the following axioms.
\be 
\item[(S1)]  $\fs(0) = [A \to A \oplus C \to C]$.
   
\item[(S2)] $\fs(\delta)\oplus \fs(\delta ') = \fs(\delta \oplus \delta ') \in \S(C \oplus C', A \oplus A')$ for $\delta \in \EE(C, A)$, $\delta' \in \EE(C', A')$.
\item[(S3)]  For all $\fs(\delta) = [A\overset{x}{\longrightarrow}B\overset{y}{\longrightarrow}C] $ and $\fs(\delta) = [A'\overset{x'}{\longrightarrow}B'\overset{y'}{\longrightarrow}C'] $ such that $a_{\ast}\delta = c^{\ast}\delta '$ for $a: A \to A'$ and $c: C \to C'$ (that is, $(a,c): \delta \to \delta '$ is a \textit{morphism of $\EE$-extensions}). There is a commutative diagram in $\C$
    \[\begin{tikzcd}[ampersand replacement = \&] A \arrow[r, "x"] \arrow[d, "a"] \& B \arrow[d, dashed, "\exists b"] \arrow[r, "y"] \& C \arrow[d, "c"]    \\ A' \arrow[r, "x'"] \& B' \arrow[r, "y'"] \& C'     \end{tikzcd}\] and in this case we say that the triple $(a,b,c)$ \textit{realises the morphism of $\EE$-extensions $(a,c): \delta \to \delta '$}.
\ee
\end{defn}

\begin{notation}
 If $\fs(\delta) = [A\overset{x}{\longrightarrow}B\overset{y}{\longrightarrow}C]$ we may also write 
    \[ \begin{tikzcd}[ampersand replacement = \&] A \arrow[r, "x"] \& B \arrow[r, "y"] \& C \arrow[r, "\delta", dashed] \& {} \end{tikzcd} \] and call this an \textit{extriangle}. In this situation we also call $x$ an \textit{$\EE$-inflation} and $y$ an \textit{$\EE$-deflation}. We may also say that $y:B \to C$ is the \textit{cone of $x$} and $x:A \to B$ is the \textit{co-cone of $y$}. 
\end{notation}

\begin{defn} \label{defn:extricat} \cite[Definition 2.12]{nakaoka2019extriangulated} A triple $(\C, \bE, \fs)$ is an \textit{extriangulated category} if it satisfies the following conditions. In this case we call the pair $(\EE, \fs)$  an \textit{external triangulation of $\C$} and $\fs$ an \textit{$\bE$-triangulation of $\C$}.
\begin{enumerate}
 \itemindent=3.5pt
   \item[(ET1)] $\bE: \C^{op} \times \C \to \mathsf{Ab}$ is an additive bifunctor.
   \item[(ET2)] $\fs$ is an additive realisation of $\bE$.
  \item[(ET3)] For all commutative diagrams with extriangles as rows
  
   \[\begin{tikzcd}[ampersand replacement = \&] A \arrow[r] \arrow[d, "a"] \& B \arrow[r] \arrow[d, "b"] \& C \arrow[r, dashed, "\delta"] \arrow[d, "\exists c"] \& {} \\ A' \arrow[r] \& B' \arrow[r] \& C' \arrow[r, dashed, "\delta'"] \& {} \end{tikzcd}\]
 there exists $c:C \to C'$ making the diagram commute and such that $(a,c): \delta \to \delta '$ is a morphism of extriangles.
 \item[ \quad \; (ET3)\textsuperscript{op}] Dual to (ET3).

   \item[(ET4)] For any pair of extriangles of the form
\[ \begin{tikzcd}[sep = scriptsize] A \ar[r, "f"] & B \ar[r, "f'"] & D \ar[r, dashed, "\delta"] & {} \end{tikzcd}, \qquad  \begin{tikzcd}[sep = scriptsize] B \ar[r, "g"] & C \ar[r, "g'"] & F \ar[r, dashed, "\delta'"] & {} \end{tikzcd} \] 
 there exists a commutative diagram 
   \[\begin{tikzcd}[ampersand replacement = \&] A \arrow[r, "f"] \arrow[d, equal] \&  B \arrow[r, "f'"] \arrow[d, "g"] \& D \arrow[d, "h"] \arrow[r, dashed, "\delta"] \& {} \\ A \arrow[r, "gf"] \& C \arrow[d, "g'"]  \arrow[r, "e"] \& E \arrow[r, "\delta''",  dashed] \arrow[d, "h'"] \& {} \\ \& F \arrow[d, "\delta'", dashed]  \arrow[r, equal] \& F \arrow[d, "(f')_{\ast}\delta'", dashed] \& \\ \& {} \& {} \& \end{tikzcd} \]
  with columns and rows being extriangles such that $h^{\ast}\delta''= \delta$ and $f_{\ast}\delta '' = (h')^{\ast} \delta'$ (that is, $(f,h'): \delta'' \to \delta'$ is a morphism of extriangles).
   \item[(ET4)\textsuperscript{op}] Dual to (ET4).
\end{enumerate}
\end{defn}

For the rest of this section, let $(\mathcal{C}, \EE, \fs)$ be an extriangulated category.

\begin{rem} \label{rem:extnclosedextri} \cite[Remark 2.18]{nakaoka2019extriangulated}   Let $\mathcal{X} \subseteq \C$ be a subcategory of $\C$ that is closed under $\EE$-extensions, that is, for all extriangles 
\[ \begin{tikzcd}[ampersand replacement = \&] A \arrow[r] \& B \arrow[r] \& C \arrow[r,  dashed] \& {} \end{tikzcd} \] the implication $A, C \in \mathcal{X} \Rightarrow B \in \mathcal{X}$ holds. Then $(\mathcal{X}, \EE|_{\mathcal{X}}, \fs|_{\mathcal{X}})$ is an extriangulated category. 
\end{rem}

\begin{lem} \label{lem:extriles} \cite[Corollary 3.12]{nakaoka2019extriangulated}
Let $ \begin{tikzcd}[ampersand replacement = \&] A \arrow[r, "x"] \& B \arrow[r, "y"] \& C \arrow[r, "\delta", dashed] \& {} \end{tikzcd}$ be an extriangle. Then there are long exact sequences 
\[ \begin{tikzcd}[ampersand replacement = \&, row sep = small] 
\C(-,A) \arrow[r, "{\C(-,x)}"] \& \C(-,B) \arrow[r, "{\C(-,y)}"] \& \C(-,C) \arrow[r, "\delta_\#"] \& 
\EE(-,A)  \arrow[r, "{\EE(-,x)}"] \& \EE(-,B) \arrow[r, "{\EE(-,y)}"] \& \EE(-,C) \\
 \C(C,-) \arrow[r, "{\C(y,-)}"] \& \C(-,B) \arrow[r, "{\C(x,-)}"] \& \C(-,A) \arrow[r, "\delta^\#"] \& \EE(C,-)  \arrow[r, "{\EE(y,-)}"] \& \EE(B,-) \arrow[r, "{\EE(x,-)}"] \& \EE(A,-)   
\end{tikzcd} \] of functors and natural transformation in the functor categories $[\C^{\text{op}}, \mathsf{Ab}]$ and $[\C, \mathsf{Ab}]$ respectively. The maps $\delta_\#$ and $\delta^\#$ are given  at $X \in \C$ by  \begin{align*}
(\delta_\#)_X : \C(X, C)& \to \EE(X,A), &(\delta^\#)_X: \C(A, X)&\to \EE(C,X) \\ f &\mapsto f^\ast \delta& g &\mapsto g_\ast \delta.
\end{align*}
\end{lem}

\begin{lem}  \cite[Proposition 1.20]{Liu2019} \label{lem:LN1.20} Let $\delta \in \EE(C, A)$ be an extriangle and $a \in \C(A, A')$ a morphism. Let $(a, b, 1)$ be a triple of morphisms realising the morphism of extriangles $(a,1): \delta \to a_{\ast}\delta$
\begin{equation} \label{weakpushoutextri}
\begin{tikzcd}[ampersand replacement = \&] A \arrow[r, "x"] \arrow[d, "a"] \& B \arrow[r, "y"] \arrow[d, "b"] \& C \arrow[r, dashed, "\delta"] \arrow[d, equal] \& {} \\ A' \arrow[r, "x'"] \& B' \arrow[r, "y'"] \& C \arrow[r, dashed, "a_\ast \delta"] \& {.} 
\end{tikzcd}
\end{equation}  
Then the sequence $\begin{tikzcd} A \arrow[r, "{\bsm -a \\ x \esm}"] & A' \oplus B \arrow[r, "{\bsm y \: b \esm}"] & C \ar[r, dashed, "y'^\ast\delta"] & {} \end{tikzcd}$ is an extriangle. 
\end{lem}

\begin{cor} \label{cor:weakpushout}
The left hand square in Diagram \ref{weakpushoutextri} is a weak pushout and weak pullback. 
\end{cor}
\begin{proof}
By Lemma \ref{lem:extriles}, every $\EE$-deflation of an extriangle is a weak cokernel of the $\EE$-inflation and vice-versa. The claim now follows from Lemma \ref{lem:LN1.20}, since the weak pushout property is equivalent to $\bsm y \: b \esm :  A' \oplus B \to C$ being a weak cokernel of $\bsm -a \\ x \esm: A \to A' \oplus B$. 
\end{proof}

\begin{lem} \cite[Corollary 3.5]{nakaoka2019extriangulated} \label{NP3.5}
Let \[\begin{tikzcd}[ampersand replacement = \&] A \arrow[r, "x"] \arrow[d, "a"] \& B \arrow[r, "y"] \arrow[d, "b"] \& C \arrow[r, dashed, "\delta"] \arrow[d, "c"] \& {} \\ A' \arrow[r, "x'"] \& B' \arrow[r, "y'"] \& C' \arrow[r, dashed, "\delta'"] \& {} \end{tikzcd}\]
be a morphism of extriangles. Then the following are equivalent
\be 
\item $a$ factors through $x$;
\item $a_\ast \delta = c^\ast \delta' = 0$;
\item $c$ factors through $y'$.
\ee
\end{lem}

\begin{defn}
Let $\C=(\C, \EE, \fs)$ be an extriangulated category. An object $P \in \C$ is \textit{$\EE$-projective} if $\EE(P, C) =0$ for all $C \in C$. By $\mathsf{Proj}_{\EE}\C$ we denote the subcategory of $\EE$-projective objects. We say that $\C$ has \textit{enough $\EE$-projectives} if for all $C \in \C$ there exists an extriangle 
 \[ \begin{tikzcd}[ampersand replacement = \&] A \arrow[r] \& P \arrow[r] \& C \arrow[r,  dashed] \& {} \end{tikzcd} \] with $P \in \mathsf{Proj}_{\EE}\C$. The notions of $\EE$-injectives and having enough $\EE$-injectives are defined dually. The subcategory of $\EE$-injective objects is denoted by $\mathsf{Inj}_{\EE}\C$. 
\end{defn}

\subsection{Right triangulated categories}

We begin by recalling in full the definition of a right triangulated (or suspended) category first introduced in \cite{Keller1987}. 

\begin{defn} \label{def:righttri}
Let $\mathcal{R}$ be an additive category and $\Sigma : \mathcal{R} \rightarrow \mathcal{R}$ an endofunctor. A \textit{right triangulation} of the pair $(\mathcal{R}, \Sigma)$ is a collection $\Delta$ of sequences of the form \[ \begin{tikzcd} A \arrow[r] & B \arrow[r]& C \arrow[r] &\Sigma A \end{tikzcd} \] in $\R$ that satisfy the following axioms.
\begin{enumerate}[label=(\alph*)]
\item[(R1)] \begin{enumerate}[label=(\roman*) ,leftmargin=*]
\item $\Delta$ is closed under isomorphisms. That is, for every commutative diagram 
\[ \begin{tikzcd} A \arrow[r] \arrow[d, "f"] & B \arrow[r] \arrow[d] & C \arrow[r] \arrow[d] &\Sigma A \arrow[d, "\Sigma f"] \\ A' \arrow[r] & B' \arrow[r] & C' \arrow[r] & \Sigma (A') \end{tikzcd} \] in $\mathcal{R}$ whose vertical arrows are isomorphisms, one row belongs to $\Delta$ if and only if the other row also belongs to $\Delta$.
\item For every $A \in \mathcal{R}$, the sequence \[ \begin{tikzcd} 0 \arrow[r] & A \arrow[r, "1_A"] & A  \arrow[r] & 0  \end{tikzcd} \] belongs to $\Delta$.
\item Every morphism $x: A \rightarrow B $ in $\mathcal{R}$ can be embedded into a sequence \[ \begin{tikzcd} A \arrow[r, "x"] & B  \arrow[r] & C_x \arrow[r] & \Sigma x \end{tikzcd} \] in $\Delta$. 
\end{enumerate}
\item[(R2)] If the sequence \[ \begin{tikzcd} A \arrow[r, "x"] & B \arrow[r, "y"]& C \arrow[r, "z"] &\Sigma A \end{tikzcd} \] is in $\Delta$ then so is the sequence \[ \begin{tikzcd} B \arrow[r, "y"]& C \arrow[r, "z"] &\Sigma A \arrow[r, "-\Sigma z"]& \Sigma B. \end{tikzcd} \]
\item[(R3)] Every commutative diagram \[ \begin{tikzcd} A \arrow[r] \arrow[d, "f"] & B \arrow[r] \arrow[d] & C \arrow[r]  &\Sigma A \arrow[d, "\Sigma f"] \\ A' \arrow[r] & B' \arrow[r] & C' \arrow[r] & \Sigma (A') \end{tikzcd} \] in $\mathcal{R}$ whose rows belong to $\Delta$ can be extended to a commutative diagram \[ \begin{tikzcd} A \arrow[r] \arrow[d, "f"] & B \arrow[r] \arrow[d] & C \arrow[r] \arrow[d, dashed] &\Sigma A \arrow[d, "\Sigma f"] \\ A' \arrow[r] & B' \arrow[r] & C' \arrow[r] & \Sigma (A'). \end{tikzcd} \]
\item[(R4)] Let \[ \begin{tikzcd} A \arrow[r, "x"] &B \arrow[r, "x'"] &X \arrow[r, "x''"] & \Sigma A,   &B \arrow[r, "y"] &C \arrow[r, "y'"] &Y \arrow[r, "y''"]& \Sigma B, \end{tikzcd}  \] and \[ \begin{tikzcd} A \arrow[r, "y  x"] &C \arrow[r, "z'"] &Z \arrow[r, "z''"] & \Sigma A\end{tikzcd} \] be sequences in $\Delta$. Then there is a commutative diagram 
\[ \begin{tikzcd} A \arrow[r, "x"] \arrow[d, equal]&B \arrow[r, "x'"] \arrow[d, "y "] &X \arrow[r, "x''"] \arrow[d, dashed, "\alpha "] & \Sigma A  \arrow[d, equal] 
				\\   A \arrow[r, "y x"] \arrow[d, "x"]&B \arrow[r, "z'"] \arrow[d, equal]&Y \arrow[r, "z''"] \arrow[d, dashed, "\beta "] &\Sigma A \arrow[d, "\Sigma x"]
					\\ B \arrow[r, "y"] &C \arrow[r, "y'"] &Z \arrow[r, "y''"] \arrow[d, dashed, "\gamma "] &\Sigma B \arrow[dl,  "\Sigma (x') "] 
					\\ 					& 					& \Sigma X 				 \end{tikzcd} \] in $\mathcal{R}$ such that the dotted column belongs to $\Delta$.
\end{enumerate} If $\Delta$ is a right triangulation of $(\R, \Sigma)$, then the triple $(\R, \Sigma, \Delta)$ is called a \textit{right triangulated category} and the sequences in $\Delta$ are called \textit{right triangles}. Following \cite{assem1998right}, if the functor $\Sigma$ is fully faithful and its image, $\Sigma \R$, is closed under extensions we call $\Sigma$ a \textit{right semi equivalence} and $\mathcal{R}$ a \textit{right triangulated category with right semi equivalence}. 
\end{defn}

\begin{defn} Let $(\R, \Sigma, \Delta)$ and $(\R', \Sigma ', \Delta ')$ be right triangulated categories. An additive functor $F: \R \to \R'$ is a \textit{right triangle functor} if \be 
\item There is a natural isomorphism $\begin{tikzcd} \zeta : F \Sigma \arrow[r, "\cong"] & \Sigma ' F \end{tikzcd}$;
\item For all right triangles $\begin{tikzcd} A \arrow[r, "x"] & B \arrow[r, "y"]& C \arrow[r, "z"] &\Sigma A \end{tikzcd}$ in $\Delta$, the sequence $$\begin{tikzcd} FA \arrow[r, "Fx"] & FB \arrow[r, "Fy"]& FC \arrow[r, "\zeta_A \circ Fz "] &\Sigma 'FA \end{tikzcd}$$ is in $\Delta '$.
\ee
\end{defn}

\begin{rem} 
Let $\R = (\R, \Sigma, \Delta)$ be a right triangulated category such that $\Sigma \R$ is closed under extensions. Then the subcategory $\Sigma \R$ naturally inherits the structure of a right triangulated category. Clearly, if $\Sigma$ is a right semi-equivalence then $\Sigma \R$ is also a right triangulated category with right semi-equivalence. We will see in Section~\ref{sec:aislesquotients} examples where $\Sigma \R$ is a right triangulated category with right semi-equivalence but $\R$ is not. 
\end{rem}

For homological properties of right triangulated categories we point the reader to \cite{assem1998right}. 


We finish this subsection by noting that there are no finite right triangulated categories with right semi-equivalence that are not triangulated. This is one of the many ways to see that any additively finite triangulated category admits no non-trivial (co-)t-structures. We refer to \cite{Amiot2007}
for a discussion of such categories.

\begin{lem} \label{lem:righttrifinite}
Let  $\R = (\R, \Sigma, \Delta)$ be a right triangulated category with right semi-equivalence. Suppose that $\R$ is idempotent complete and additively finite, then $\Sigma$ is an autoequivalence. In particular,  $\R = (\R, \Sigma, \Delta)$ is a triangulated category.
\end{lem} 
 \begin{proof}
 Let $S=\{X_1, \dots, X_n\}$ be a set of isomorphism classes of indecomposable objects in $\R$. Since $\Sigma$ is additive, fully faithful and $\R$ is idemopotent complete. $\Sigma$ acts as a permutation on the set $S$ and we deduce that $\Sigma^m = 1_\R $ for some $m \in \mathbbm{N}$. 
 \end{proof}

\subsubsection{(Co-)stabilisation of a right triangulated category} \label{section:costab}
 
We recall the definitions of two triangulated categories associated to a right triangulated category: the stabilisation and costabilisation.
 
The \textit{stabilisation of a right triangulated category}, $\R=(\R, \Sigma, \Delta)$, consists of a pair $(\SR, s)$ where $\SR$ is a triangulated category and $s: \R \to \SR$ is a right triangle functor satisfying a universal property: For all right triangle functors $F: \R \to \T$ with $\T$ being a triangulated category, there exists a unique triangle functor $F': \SR \to \T$ such that $F's=F$
\[ \begin{tikzcd} \R \ar[r, "s"] \ar[d, "\forall F"'] & \SR \ar[ld, "\exists ! F'"] \\ \T. & \end{tikzcd}\]

Dually, the \textit{co-stabilisation of $\R$} consists of a pair $(\C(\R), c)$ where $\C(\R)$ is a triangulated category and $c:C(\R) \to \R$ is a right triangle functor satisfying a universal property:
For all right triangle functors $G:\T \to \R$ with $\T$ being a triangulated category, there exists a unique triangle functor $G':\T \to \C(\R)$ such that $G=cG'$.  

The stabilisation and co-stablisation of a right triangulated category always exist. Since we will use it explicitly, we recall the construction of the stablisation from \cite[Section 3.1]{Beligiannis2000} and \cite{Heller1968}; see also \cite{Freyd1966}. For more information on the co-stabilisation, which may be constructed as the category of spectra, we refer to \cite[Section 4.5-4.7]{Grandis1995}; see also \cite{Beligiannis2000, Jorgensen2001}.
 
 Let $\R=(\R, \Sigma, \Delta)$ be a right triangulated category.

\begin{defn}
We define the additive category $\SR$ as follows.  The objects of $\SR$ are pairs $(A, n)$ with $A \in \R$ and $n \in \mathbb{Z}$. The spaces of morphisms are given by 
\[ \SR \left[ (A, n), (B,m)\right] =  \underset{\longrightarrow k \in J}{\mathsf{colim}}{\left(\R( \Sigma^{n-k}A, \Sigma^{m-k}B)\right)}, \; J = \{k \in \mathbb{Z} \mid k \leq \mathsf{min}\{n,m\} \}.  \]
There is an autoequivalence of $\SR$ which is given on objects by $\Sigma(A,n) = (A, n+1)$ and induced on morphisms by the natural map $$ \R( \Sigma^{n-k} A, \Sigma^{m-k}B) \to \R( \Sigma^{n+1-k} A, \Sigma^{m+1-k} B)$$ for all $k \leq \mathsf{min}\{m, n\}$. By abuse of notation, we denote this autoequivalence also by $\Sigma$. 

The functor $s: \R \to \SR$ is given on objects by $s(A) = (A, 0)$ and for a morphism $f:A \to B$, $s(f:A \to B)$ is the zero-representative of $\SR((A,0), (B,0))$. 

Triangles in $\SR$ are given by sequences isomorphic to sequences of the form 
\[ \begin{tikzcd} (A,n) \ar[r, "x"] & (B,m) \ar[r, "y"] & (C, l) \ar[r, "z"] & (A, n+1) \end{tikzcd} \] such that there exists $k \leq \mathsf{min}\{n, m, l\}$ such that

 \begin{equation} \label{ST triangles}
 \begin{tikzcd}[sep=huge] \Sigma^{n-k} A \ar[r, "(-1)^k\Sigma^{-k}x"] & \Sigma^{m-k}B \ar[r, "(-1)^k\Sigma^{-k}y"] & \Sigma^{l-k} C \ar[r, "(-1)^k\Sigma^{-k}z"] & \Sigma^{n+1-k}A \end{tikzcd} 
\end{equation} is a right triangle in $\R$.  

\end{defn}

One may verify that $S\Sigma \cong \Sigma S$ and that $S$ is a right triangle functor. 

\begin{rem} We make some observations. \be \item 
If the endofunctor $\Sigma: \R \to \R$ is fully faithful, then the morphism spaces in $\SR$ are neater
\[ \SR \left[ (A, n), (B,m)\right] \cong \R(\Sigma^{n-k}A, \Sigma^{m-k} B), \; \forall k \leq \mathsf{min}\{n, m\}. \]
 Further, in this case, $s:\R \to \SR$ is fully faithful and  $\SR$ is the smallest triangulated category that contains $\R$ as a full right triangulated subcategory.
 
 \item If $\Sigma :\R \to \R$ is a right semi-equivalence then the `there exists' preceding Equation (\ref{ST triangles}) can be replaced by `for all'. 
 
 \item  Every indecomposable object in $\SR$ is isomorphic to an object of the form $(A, n)$ with $A \in (\R \setminus \Sigma \R) \cup \{0\}$.

 \ee
\end{rem}

Many categorical properties of $\SR$ are inherited from $\R$, for instance if $\R$ has (co-)products then so does $\SR$. The following properties will be useful for our work.

\begin{lem} \label{lem:bounded}
Suppose that for all $A, B \in \R$ we have that $\R(A, \Sigma^i B ) =0 $ for $i>>0$. Then, for all $X, Y \in \SR$ we have that $\SR(X, \Sigma^i Y) =0$ for all $i>>0$. 
\end{lem}
\begin{proof}
Let $X, Y \in \SR$ then there exists $A, B \in \R$ and $k >0$ such that $\Sigma^k X \cong A$ and $\Sigma^k Y \cong B$. Therefore, for $i>>0$, we have that \[ \SR(X, \Sigma^i Y) \cong \SR(A, \Sigma^i B) \cong \R(A, \Sigma^i B) = 0.\]
\end{proof}

We call a right triangulated category satisfying the condition of the above Lemma \ref{lem:bounded} \textit{bounded.}

\begin{lem} \label{lem: T extn closed}
Let $\R=(\R, \Sigma, \Delta)$ be a right triangulated category with right semi-equivalence. Then $\R$ is extension closed as a subcategory of $\SR$. 
\end{lem} 
\begin{proof}
Let \[ \begin{tikzcd} (A,n) \ar[r, "x"] & (B,m) \ar[r, "y"] & (C, l) \ar[r, "z"] & (A, n+1) \end{tikzcd} \] be a triangle in $\SR$ with $n$, $m$, $l$, minimal such that $\Sigma^{-n} A$, $\Sigma^{-m} B$, $\Sigma^{-l}C \in \R$. We must show that if $n, l = 0$ then $m = 0$.  For that purpose, suppose that $m<0$, then $\mathsf{min}\{m,n,l\} = m$ and, by definition, there is a right triangle in $\R$ 
\[ \begin{tikzcd}[sep=huge] \Sigma^{n-m} A \ar[r, "(-1)^m\Sigma^{-m}x"] &B \ar[r, "(-1)^m\Sigma^{-m}y"] & \Sigma^{l-m} C \ar[r, "(-1)^m\Sigma^{-m}z"] & \Sigma^{n+1-m}A. \end{tikzcd} \] As $n-m, l-m>0$ and $\Sigma: \R \to \R$ is a right semi-equivalence, we deduce that $B \in \Sigma \R$ which is a contradiction to the minimality of $m$. 
\end{proof}

\begin{exmp}
Let $\mathcal{T}$ be a triangulated category and $\R \subset \T$ a subcategory that is closed under positive shifts and extensions. Then $\R$ is a right triangulated category with right semi-equivalence and $\SR \cong \mathsf{cosusp}_{\T}(\R)$, the smallest subcategory of $\T$ containing $\R$ that is closed under negative shifts and extensions.
\end{exmp}

For more examples of the stabilisation of right triangulated categories, see \cite[Section 3.1]{Beligiannis2000} and \cite{Keller1987}.

\section{Right triangulated categories as extriangulated categories} \label{sec:rightriextri}

The aim of this section is show that right triangulated categories have a natural extriangulated structure precisely when the shift functor is a right semi-equivalence. We also describe which extriangulated categories have a right triangulated structure. To do this, we begin by using relative homological algebra to define new extriangulated structures from existing ones (Section \ref{S3.1}). We then show when these new extriangulations induce right triangulated structures on a quotient category (Section \ref{S3.2}). With these tools in hand we complete the above aims in Section \ref{S3.3}. 

\subsection{Extriangulated structures using relative homological algebra} \label{S3.1}
 
 In this section we show how one can use relative homological algebra to construct new extriangulated structures and characterise when these exact structures have enough injectives/ projectives. The existence of these extriangulated structures also follows from \cite[Proposition 3.17]{herschend2017n} for $n=1$ but we give an alternative proof using relative notions. We begin with an easy lemma.

\begin{lem}
Let $(\mathcal{C}, \EE, \fs)$ be an extriangulated category, $\mathcal{D}$ be a class of objects in $\mathcal{C}$ and $\delta$ be an extriangle. Then the following are equivalent
\be
\item There exists representative of $\fs(\delta)$, $A\overset{x}{\longrightarrow}B\overset{y}{\longrightarrow}C$, such that $x$ is $\D$-monic;
\item The $\EE$-inflation of every representative of $\fs(\delta)$ is $\D$-monic.   
\ee
\end{lem}
\begin{proof}
(a)$\Rightarrow$(b) follows from Lemma \ref{biglez}(c). (b)$\Rightarrow$(a) is obvious. 
\end{proof}

\begin{prop} \label{prop:new exact str}
Let $(\mathcal{C}, \EE, \fs)$ be an extriangulated category and $\mathcal{D}$ a class of objects in $\mathcal{C}$. Consider the classes of extriangles for $A, C \in \C$
\begin{align*}
	\II_\mathcal{D}(C,A) &= \{ \begin{tikzcd}[ampersand replacement = \&, sep =small] A \arrow[r, "x"] \& B \arrow[r, "y"] \& C \arrow[r, "\delta", dashed] \& {} \end{tikzcd} \in \EE(C, A)  \mid x \in \mathsf{Mon}(\mathcal{D}) \}, \\
	 \PPP_\mathcal{D}(C,A) &= \{ \begin{tikzcd}[ampersand replacement = \&, sep =small] A \arrow[r, "x"] \& B \arrow[r, "y"] \& C \arrow[r, "\delta", dashed] \& {} \end{tikzcd} \in \EE(C, A) \mid y \in \mathsf{Epi}(\mathcal{D}) \}, \text { and} \\ 
	 	\DD_\mathcal{D}(C,A) &= \{ \begin{tikzcd}[ampersand replacement = \&, sep =small] A \arrow[r, "x"] \& B \arrow[r, "y"] \& C \arrow[r, "\delta", dashed] \& {} \end{tikzcd} \in \EE(C, A) \mid x \in \mathsf{Mon}(\mathcal{D}), y \in \mathsf{Epi}(\mathcal{D}) \} = \II_\mathcal{D} \cap \PPP_\mathcal{D}.
\end{align*} 
Then $(\II_\D, \fs|_{\II_\D}), (\PPP_\D, \fs|_{\PPP_\D})$ and $(\DD_{\D}, \fs|_{\DD_\D})$ all define external triangulations of $\mathcal{C}$. Moreover, 
\begin{align*}
\mathcal{D} &\subseteq \mathsf{Inj}_{\II_\mathcal{D}} \C \subseteq \mathsf{Inj}_{\DD_\mathcal{D}}\C  \text { and} \\ \mathcal{D} &\subseteq \mathsf{Proj}_{\PPP_\mathcal{D}} \C  \subseteq \mathsf{Proj}_{\DD_\mathcal{D}}\C. 
\end{align*}
Additionally, suppose that $\mathcal{D}= \mathsf{Add}(\mathcal{D})$ then the following hold. 
\be \item If $\mathcal{D}$ is covariantly finite and left $\mathcal{D}$-approximations are $\EE$-inflations  then $\mathcal{D} =  \mathsf{Inj}_{\II_\mathcal{D}} \C$ and  $\C$ has enough $\II_\mathcal{D}$-injectives.
\item If $\mathcal{D}$ is contravariantly finite and right $\mathcal{D}$-approximations are $\EE$-deflations then $\mathcal{D} = \mathsf{Proj}_{\PPP_\mathcal{D}} \C$ and $\C$ has enough $\PPP_\mathcal{D}$-projectives.
\item If $\mathcal{D}$ is functorially finite, left (resp. right) $\mathcal{D}$-approximations are $\EE$-inflations (resp. $\EE$-deflations) and cones (resp. co-cones)  of left (resp. right) $\mathcal{D}$-approximations are $\mathcal{D}$-epic (resp. $\mathcal{D}$-monic) then $\mathcal{D} = \mathsf{Inj}_{\DD_\mathcal{D}}\C=\mathsf{Proj}_{\DD_\mathcal{D}}\C$ and $\C$ has enough $\DD_\mathcal{D}$-injectives and $\DD_\mathcal{D}$-projectives. 
\ee
\end{prop}
\begin{proof}
We prove the statements for $\II_\mathcal{D}$ whence the remaining claims follow from dual and combined arguments. Let $\delta \in \II_\D(C', A)$, $a \in \C( A, A')$ and $c \in \C(C, C')$. By \cite[Proposition 3.14]{herschend2017n} it is enough to show that $a_{\ast} \delta \in \II_\D(C', A'), c^{\ast}\delta \in \II_\D(C, A)$ and $\II_\D$-inflations are closed under composition. The first follows from Corollary  \ref{cor:weakpushout} and Lemma \ref{biglez}(c) and the third from Lemma \ref{biglez}(a). To show the second claim consider the morphism of extriangles $(1, c): c^{\ast}\delta \to \delta$ 
\[
\begin{tikzcd}[ampersand replacement = \&] A \arrow[r, "x"] \arrow[d, equal] \& B \arrow[r, "y"] \arrow[d, "b"] \& C \arrow[r, dashed, "c^\ast \delta"] \arrow[d, "c"] \& {} \\ A \arrow[r, "x'"] \& B' \arrow[r, "y'"] \& C' \arrow[r, dashed, "a_\ast \delta"] \& {.} 
\end{tikzcd}
\]
We see that $x' = bx \in \Mon (\D)$ and hence $x$ is $\D$-monic by Lemma \ref{biglez}(b). 

We verify the additional claims. By construction we have that the class of $\D$-monic morphisms is contained in the class of $\II_\D$-inflations, $\omega$. Thus 
$$ \D \subseteq \mathsf{Inj}(\mathsf{Mon}(\D)) \subseteq \mathsf{Inj}(\omega) = \mathsf{Inj}_{\II_\D}\C.
$$
Now additionally suppose that $\D = \mathsf{Add}(\D)$, $\mathcal{D}$ is covariantly finite and left $\mathcal{D}$-approximations are $\EE$-inflations. Then for every object $I \in \mathsf{Inj}_{\II_\D}\C$ there is an $\II_\D$-inflation $i:I \to D$ with $D \in \D$. We deduce that $I$ is a direct summand of $\D$ and we are done. 
\end{proof}

\begin{notation} Let $(\mathcal{C}, \EE, \fs)$ be an extriangulated category.
For brevity, in the sequel we will say a subcategory $\D$ of $\C$ has property $(\ast )$ if $\mathcal{D}= \mathsf{Add}(\mathcal{D})$, $\D$ is covariantly finite and all left $\D$-approximations are $\EE$-inflations.
\end{notation}

We note that one may view the Frobenius property \cite{Heller1960, happel1988triangulated} of external triangulations through the lens of relative homological algebra.

\begin{cor} \label{cor:frobextri}
An extriangulated category $(\mathcal{C}, \EE, \fs)$ is Frobenius if and only if there exists a functorially finite subcategory $\PP$ such that 
\be \item All left $\PP$-approximations are $\EE$-inflations and all $\EE$-inflations are $\PP$-monic;
\item  All right $\PP$-approximations are $\EE$-deflations and all $\EE$-deflations are $\PP$-epic.
\ee In other words, $\PP = \mathsf{Proj}_{\DD_\PP}\C = \mathsf{Inj}_{\DD_\PP}\C$ and $\DD_\PP = \EE$. 
\end{cor}

\subsection{Right triangulated stable categories} \label{S3.2}

We show how extriangulations can induce right triangulated structures on a quotient category. The construction is reminiscent of the triangulated structure of the stable category of a Frobenius exact category \cite{Heller1960, happel1988triangulated}. 

\textbf{Set-up:} For the rest of this section, let $(\mathcal{C}, \EE, \fs)$ be an extriangulated category and $\D \subseteq \C$ be a subcategory satisfying property $(\ast)$. For each $A \in \C$ we make a choice of extriangle 
\begin{equation}
\begin{tikzcd} A \ar[r, "i_A"] & \D (A) \ar[r, "p_A"] & \Sigma_\D A \ar[r, "\delta_A", dashed] &{} \end{tikzcd} 
\end{equation} where $i_A$ is a left $\D$-approximation of $A$. We also define $\Sigma_D$ on morphisms: Let $f: A \to B$. Then, as $i_A$ is $\D$-monic, there exists $\D(f): \D(A) \to \D(B)$ such that $\D(f)i_A = i_B f$ and thus, by (ET3) there exists a morphism $\SD f: \SD A \to \SD B$ such that $(f, \SD): \delta_A \to \delta_B$ is a morphism of extriangles. 
 \begin{equation}
\begin{tikzcd} A \arrow[r, "i_A"] \arrow[d, "f"] & \D(A) \arrow[d, dashed, " \exists"', "\D(f)"] \arrow[r, "p_A"] & \SD A \arrow[d, dashed, "\exists"',  "\SD f"] \ar[r, "\delta_A", dashed] &{} \\ B \arrow[r, "i_B"'] & \D(B)\arrow[r, "p_B"'] & \Sigma_\D B \ar[r, "\delta_B", dashed] & {.}  \end{tikzcd} 
 \end{equation}
By $\underline{\C}_\D$ we denote the stable category of $\C$ by the ideal of morphisms factoring through objects of $\D$. 

The following lemma is similar to \cite[Claim 6.1]{nakaoka2019extriangulated}.

\begin{lem} \label{lem:shiftSD}
The construction above of $\SD$ defines an additive endofunctor $\CD \to \CD$. Further, any other choice of extriangles $\delta_A$ yields a naturally isomorphic endofunctor.
\end{lem}
\begin{proof} We keep the notation of the above paragraph. 
First we show that $\SD$ is well-defined on morphisms in $\CD$, that is, the construction of $\SD f$ is independent of the choices made of $\D(f)$ and in (ET3). Indeed, let $g: \SD A \to \SD B$ be any morphism such that $(f,g): \delta_A \to \delta_B$ is a morphism of extriangles.  Then $g^{\ast}\delta_B = f_{\ast}\delta_A = (\SD f)^\ast \delta_B$ and therefore $(g-\SD f)^\ast \delta_B =0$. By Lemma \ref{NP3.5}, $g-\SD f$ factors through $p_B$ and hence $\und{g-\SD f} =0$ in $\CD$.

 Now suppose for all $A \in \C$ we have another choice of extriangle 
\[ \begin{tikzcd} A \ar[r, "i'_A"] & \D' (A) \ar[r, "p'_A"] & \Sigma_\D' A \ar[r, "\delta'_A", dashed] &{} \end{tikzcd} \] which then results in another endofunctor $\SD' : \CD \to \CD$. Then, since $i_A$ is $\D$-monic, there exists $s:\D(A) \to \D'(A)$ such that $si_A = i'_A$. Then, by (ET3) there exists $t_A =t:\SD A \to \SD' A$ such that $(1, t): \delta_A \to \delta'_A$ is a morphism of extriangles. Similarly, we obtain a morphism $t'_A=t': \SD' A \to \SD A $ such that $(1, t'): \delta'_A \to \delta_A$ is a morphism of extriangles 
\[ 
\begin{tikzcd} A \arrow[r, "i_A"] \arrow[d, equal] & \D(A) \arrow[d, dashed, " \exists"', "s"] \arrow[r, "p_A"] & \SD A \arrow[d, dashed, "\exists"',  "t"] \ar[r, "\delta_A", dashed] &{} \\ A \arrow[r, "i'_A"] \ar[d, equal] & \D'(A)\arrow[r, "p'_A"] \arrow[d, dashed, " \exists"'] & \Sigma_\D' A \ar[r, "\delta'_A", dashed] \ar[d, dashed, "\exists"', "t'"] & {} \\ A \arrow[r, "i_A"]  & \D(A) \arrow[r, "p_A"] & \SD A  \ar[r, "\delta_A", dashed] &{.}  \end{tikzcd} 
\]
We claim that $\{t_A\}_{A \in \CD}$ is an isomorphism of functors $\SD \to \SD'$ with inverse given by $\{t'_A\}_{A \in \CD}$. The fact that the $t_A$ are isomorphisms in $\CD$ follows from the observation that $(1, t't): \delta_A \to \delta_A$ is a morphism of extriangles and hence, by using a similar argument to the above, we see that $\und{t't} = \und{1_{\SD A}}$ in $\CD$. Dually, $\und{tt'} = \und{1_{\SD'A}}$. 

It remains to verify that $\{t_A\}_{A \in \CD}$ is a natural transformation. Let $f:A \to B$ we must show that $\und{t_B \SD f} = \und{\SD'f t_A}$. This follows from the observation that the pairs $(f, ( t_B \SD f))$ and $(f, (\SD'f t_A))$ both define morphisms of extriangles $\delta_A \to \delta'_B$. 
\end{proof}

\begin{prop} \label{thm:abm} 
The stable category $\underline{\mathcal{C}}_\mathcal{D}$ with the endofunctor $\Sigma_\D$ admits a right triangulation given by the collection of all sequences isomorphic to sequences of the form 
 \[ \begin{tikzcd} A \arrow[r, "\underline{f}"] & B \arrow[r, "\underline{g}"]& C \arrow[r, "\underline{h}"] & \Sigma x \end{tikzcd} \]
that fit into a commutative diagram in $\mathcal{C}$
 \[ \begin{tikzcd} A \arrow[r, "i_A"] \arrow[d, "f"'] & \D(A) \arrow[r, "p_A"] \arrow[d] & \Sigma_\D A \arrow[d, equal] \arrow[r, dashed, "\delta_A"] & {} \\ B \arrow[r, "g"] & C \arrow[r, "h"] & \SD A \ar[r, "f_\ast \delta_x", dashed] & .\end{tikzcd} \] 
Furthermore, $\SD \underline{\C}_\D$ is always closed under extensions and  $\SD$ is fully faithful if and only if $\mathcal{D} \subset \mathsf{Proj}_{\II_\mathcal{D}} \C$.

\end{prop}
\begin{proof}
The arguments of \cite[Theorem 3.3]{assem1998right} and \cite[Theorem 3.1]{beligiannis1994left} may be recycled to the extriangulated setting.
\end{proof}

We can precisely describe when $\Sigma_\D$ is a right semi-equivalence.

\begin{prop} \label{funcyisom}
Suppose additionally that cones of left $\D$-approximations are $\D$-epic. Then the following hold. 
\be 
\item $\II_\D = \DD_\D$. 
\item $\underline{\C}_\D$ is a right triangulated category with right-semi equivalence.
\item 
For all $A, C \in \mathcal{C}$ there is a functorial isomorphism of abelian groups \begin{align*}
F = F_{C,A}:  \und{\C}_\D (C, \Sigma_\D A ) & \overset{\cong}{\longrightarrow}  \DD_\D(C, A) \\  \und{f} &\longmapsto f^{\ast}\delta_A.
\end{align*}
\ee 
\end{prop}
\begin{proof}
Claims (a) and (b) follow directly from Propositions \ref{prop:new exact str} and \ref{thm:abm}. It remains to show (c): 

\textbf{$F$ is well-defined:} We must check that if $\und{f} = \und{g} \in \CD (C, \SD A)$ then $f^\ast \delta_A = g^\ast \delta_A$. Indeed, in this case, $\und{f-g} =0$ and, since $p_A$ is $\D$-epic by part (a), $f-g$ factors through $p_A$. Now, by Lemma \ref{NP3.5}, $(f-g)^\ast \delta_A =0$ and we are done.  Note that this also shows the injectivity of $F$, since $f^* \delta_A = 0$ implies that $f$ factors through $p_A$ and thus $\und{f} =0$.

\textbf{$F$ is bijective:}
It remains to show that $F$ is surjective. Let $\gamma \in \DD_\D(C, A)$ be realised by $\begin{tikzcd} A \arrow[r,  "x"]  & B  \arrow[r,  "y"] & C \end{tikzcd}$. Then, since $x$ is $\D$-monic there exists $g:B \to \D(A)$ such that $gx = i_A$. By (ET3), there then exists $f: C \to \Sigma_\D A$ such that $(1, f): \gamma \to \delta_A$ is a morphism of extriangles. In other words, $\gamma = f^{\ast}\delta_A =:F(f)$.
\[\begin{tikzcd} A \arrow[r, "x"] \arrow[d, equal] & B \arrow[d, dashed, " \exists g"] \arrow[r, "y"] & C \arrow[d, dashed, "\exists f"] \ar[r, "f^{\ast}\delta_A", dashed] &{} \\ A \arrow[r, "i_A"'] & \D(A)\arrow[r, "p_A"'] & \Sigma_\D A \ar[r, "\delta_A", dashed] & {.}  \end{tikzcd} \]

\textbf{$F$ is a homomorphism of abelian groups:} This is straightforward: 
\[ F(f + f') = (f+f')^{\ast} \delta_A = f^\ast \delta A + f'^\ast \delta_A = F(f) + F(f'). \]

\textbf{Functorality in the first argument:} Let $\und{c} \in \C(C,C')$. Then for all $\und{f} \in \und{\C}_\D(C', \Sigma_A)$
\begin{align*}
F_{C,A} \und{\C}_\D(\und{c}, \Sigma_\D A)&: \und{f} \longmapsto \ \und{fc} \ \ \, \longmapsto  (fc)^\ast \delta_A \\ 
 \DD_\D(c, A) F_{C',A}&:  \und{f} \longmapsto  f^{\ast} \delta_A \longmapsto c^{\ast} f^{\ast} \delta_A = (fc)^\ast \delta_A
\end{align*} which proves the claim.  

\textbf{Functorality in the second argument:}  Let $\und{a} \in \CD(\Sigma_D A, \SD A')$. Since $\SD$ is full by part (b), there exists $\und{\alpha} \in \CD( A, A')$ such that $\SD \und{\alpha} = \und{a}$. Then for all $\und{f} \in \CD(C, \SD A)$
\begin{align*}
F_{C,A'} \und{\C}_\D(C, \und{a})&:  \und{f} \longmapsto \ \und{af} \ \ \, \longmapsto (af)^\ast \delta_A \\ 
\DD_\D(C, \alpha) F_{C,A}&:  \und{f} \longmapsto f^{\ast} \delta_A \longmapsto \, \alpha_{\ast} f^{\ast} \delta_A = (f)^\ast \alpha_{\ast} \delta_A = (f)^\ast a^{\ast} \delta_A
\end{align*} 
where the last equality in the second line follows from the definition of $\SD$. 
\end{proof}

In light of Corollary \ref{cor:frobextri}, extriangulated categories $(\C, \EE, \fs)$ such that the injective stable category $\und{C}_{\mathsf{Inj} \EE}$ is a right triangulated with right semi-equivalence have a `one-sided Frobenius' property: There are enough injectives and each injective object is projective. This imbalance of projectives and injectives and also Lemma~\ref{lem:righttrifinite} indicate that we must look in extriangulated categories with infinitely many objects for  examples of quotients that are right triangulated with right semi-equivalence.

\begin{exmp}
\be 
\item Let $Q$ be the infinite quiver 
\[ 1 \leftarrow 2 \leftarrow 3 \leftarrow \dots \]
and consider the category $\A = \mathrm{mod}KQ / \mathsf{rad}^m$ for some $m>1$. $\A$ is an abelian category and with it's maximal exact structure it is an extriangulated category. Observe that $\A$ has enough injectives and that each injective object is projective. Indeed, $I_r = P_{r+m-1}$. But not every projective object is injective, for instance $P_1 = S_1$ is not injective. Thus, by Theorem \ref{thm:abm}, the quotient category $\underline{\A}_{\mathsf{Inj}\A}$ is a right triangulated category with semi-equivalence with the shift  given by the co-syzygy functor. 

\item  Let $\C = (\C, \EE, \fs)$ be an extriangulated category and let $(\U, \V)$ be a cotorsion pair in $\C$ \cite[Definition 4.1]{nakaoka2019extriangulated}. It is easily verified that $\V$ satisfies property $(\ast)$. Thus $\underline{\C}_\V$ is a right triangulated category.  We will investigate examples of this flavour for the case of Frobenius extriangulated categories further in Section~\ref{sec:aislesquotients}. Let us note that the class of subcategories satisfying property $(\ast)$ more general than the class of cotorsion pairs, since the subcategories giving cotorsion pairs  must be closed under extensions.

\item  Let $\T$ be a compactly generated triangulated category. Recall that a subcategory $\X$ of $\T$ is \textit{definable} if there is a class of morphisms $\omega$ between compact objects in $\T$ such that $\X = \mathsf{Inj}(\omega)$ \cite[Section 4.1]{AngeleriHuegel2017}; where it was also shown that every definable category admits left approximations. Thus, since in a triangulated category every morphism is an $\EE$-inflation, it follows that $\underline{\T}_\X$ is a right triangulated category. 

\ee The author would be interested to know if is possible to classify when the examples in (b) and (c) that result in right triangulated categories with right semi-equivalence.   

\end{exmp}
\subsection{Right triangulated extriangulated categories} \label{S3.3}

We  characterise right triangulated categories as extriangulated categories. We begin with some terminology and a useful lemma.

\begin{defn}
Let $(\R, \Sigma, \Delta)$ be a right triangulated category. We say that the right triangulation $\Delta$ \textit{induces an extriangulated structure on $\R$} if there exists an external triangluation $(\EE, \fs)$ of $\R$ such that for all right triangles $A\overset{x}{\longrightarrow}B\overset{y}{\longrightarrow}C \overset{z}{\longrightarrow} \Sigma A$, there is an extriangle $A\overset{x}{\longrightarrow}B\overset{y}{\longrightarrow}C \overset{\delta}{\dashrightarrow}$.
\end{defn}

\begin{lem} \label{lem:np3.30} \cite[Proposition 3.30]{nakaoka2019extriangulated}  Let  $(\C, \EE, \fs)$ be an extriangulated category and $\D \subseteq \mathsf{Proj}_\EE (\C) \cap \mathsf{Inj}_\EE (\C)$ be a full, additive, replete subcategory. Then the stable category $\CD$ inherits an external triangulation, $(\und{\EE}, \und{\fs})$, given by 
\be 
\item  $\und{\EE}(C,A) = \EE(C, A)$ for all $A, C \in \C$;
\item $\und{\EE}(\und{c}, \und{a}) = \EE(c, a)$ for all $a \in \C(A, A')$, $c \in \C(C', C)$;
\item $\und{\fs}(\delta) = [\begin{tikzcd}[cramped, sep=small] A \ar[r, "\und{x}"] & B \ar[r, "\und{y}"] & C \end{tikzcd}]$ 
where $\fs(\delta) = [\begin{tikzcd}[cramped, sep=small] A \ar[r, "x"] & B \ar[r, "y"] & C \end{tikzcd}]$ for all extriangles $\delta$.
\ee 
\end{lem}

\begin{exmp} \label{ex:tricatsextri} We give two important classes of examples. \be

\item Let $\T$ be a triangulated category. Then the triangulation of $\T$ is a right triangulation which induces an extriangulated structure on $\T$. See \cite[Section 3.3]{nakaoka2019extriangulated} This was a motivating example for the introduction of extriangulated categories.

\item  Let  $(\C, \EE, \fs)$ be an extriangulated category and $\D \subset \C$ be a subcategory satisfying property $(\ast)$ and such that all cones of left $\D$-approximations are $\D$-epic.   By Proposition \ref{funcyisom}(a) we have that $\D \subseteq \mathsf{Proj}_\EE (\C) \cap \mathsf{Inj}_\EE (\C)$ and it follows from part (c) of the same result that the right triangulated structure on $\CD$ of Proposition \ref{thm:abm} coincides with the extriangulated structure of Lemma \ref{lem:np3.30}. 
 In other words, the right triangulation of $\CD$ (with right semi-equivalence) induces an extriangulation on $\CD$.
 Let us note that this could also be deduced from a combination of Lemma \ref{lem: T extn closed}, Remark \ref{rem:extnclosedextri} and the above Example.
\ee
 \end{exmp}

We may now state and prove the main result of this section.

\begin{thm} \label{thm:nicerighttriextri}
Let $(\mathcal{C}, \EE, \fs)$ be an extriangulated category. Then the following are equivalent 
\be \item There exists a fully faithful additive endofunctor $\Sigma : \C \to \C$ such that $\EE(-, ?) \cong \C (-, \Sigma ?)$ and that the image $\Sigma \C $ is closed under $\EE$-extensions;  
\item $\mathsf{Inj}_\EE (\C) = \{0 \}$ and there are enough $\EE$-injectives;
\item  There is a right triangulation of $\C$ that induces the extriangulated structure $(\EE, \fs)$.  
\ee
\end{thm}
\begin{proof}
\textbf{(a)$\Rightarrow$(b):} We claim that, for each $A \in \C$, the morphism $1_{\Sigma A} \in \C(\Sigma A, \Sigma A) \cong \EE(A, \Sigma A)$ is realised by the sequence $[A \to 0 \to \Sigma A]$. Indeed, we may use a similar argument to that of \cite[Lemma 3.21]{nakaoka2019extriangulated}: Let $\fs (1_{\Sigma A} )=  [\begin{tikzcd}[cramped, sep=small] A \ar[r, "x"] & E \ar[r, "y"] & \Sigma A \end{tikzcd} ]$, then by Lemma \ref{lem:extriles} there is a long exact sequence in $[\C^{\text{op}}, \mathsf{Ab}]$ 
\[ \begin{tikzcd}[ampersand replacement = \&] \C(-,A) \arrow[r, "{\C(-,x)}"] \& \C(-,E) \arrow[r, "{\C(-,y)}"] \& \C(-,\Sigma A) \arrow[r, "{(1_{\Sigma A})_\# = \text{id}}"] \& 
\C(-,\Sigma A)  \arrow[r, "{\C(-,\Sigma x)}"] \& \C(-,\Sigma E). \end{tikzcd}\] It follows that $ y=0= \Sigma x$. Thus $x=0$ since $\Sigma $ is faithful. Now the exactness of $0 \to \C(-,E) \to 0$ implies that $E \cong 0$.  Thus, if $I$ is an $\EE$-injective object then it is a direct summand of $0$ and so $I=0$. 

\textbf{(b)$\Rightarrow$(a),(c):} In this case, the subcategory $\{0\}$ satisfies $\EE = \II_{\{0\}}=\DD_{\{0\}}$. Thus the stable category $\und{\C}_{\{0\}} \cong \C$ has a right triangulated structure with right semi-equivalence by Propositions \ref{thm:abm} and \ref{funcyisom}. The claims follow from Example \ref{ex:tricatsextri}(b).

\textbf{(c)$\Rightarrow$(b):} By the axiom (R3)(iii), for all $A \in \C$, the morphism $A \to 0$ is the first morphism in a right triangle. Thus, by assumption, $A \to 0$ is an $\EE$-inflation and the claim follows.
\end{proof}

As a direct consequence, we see that, in general, right triangulated categories do not have a natural extriangulated structure. 

\begin{cor} \label{cor:righttriextri}
Let $(\R, \Sigma, \Delta)$ be a right triangulated category. Then $\Delta$ induces an extriangulated structure on $\R$ if and only if $\Sigma$ is a right semi-equivalence. 
\end{cor}

\begin{rem} \label{rem:negative extns}
The concept of negative (first) extensions of an extriangulated category has been recently introduced and studied \cite{Adachi2021, Gorsky2021}. For a triangulated category $\T$ one may take $\EE^{-1}(-,?) = \T(-, \Sigma^{-1} ?)$ as a negative first extension. Since by Lemma \ref{lem: T extn closed} a right triangulated category with right semi equivalence $\R$ is an extension closed subcategory of the triangulated category $\SR$ there is a natural first negative extension structure on $\R$ given by  $\EE^{-1}(C,A) := \SR(C, \Sigma^{-1}A) \cong \R(\Sigma C, A)$ for all $A, C \in \R$.
\end{rem}

\section{Aisles and co-aisles} \label{sec:aislescoaisles}

 In this section, we show that the language of extriangulated categories allows us, under some assumptions, to give an intrinsic characterisation of 
which right triangulated categories with right semi-equivalence occur as (co)-aisles of (co-)t-structures in its stabilisation. Let  $\R = (\R, \Sigma, \Delta)$  be a right triangulated category with right semi-equivalence.

\subsection{Torsion pairs}
 We begin by recalling some important definitions. 
 
\begin{defn} \label{defn:torsion} Let $\C= (\C, \EE, \fs)$ be an extriangulated category. 
A pair of additive subcategories, $(\mathcal{U}, \mathcal{V})$, of $\C$ is a \textit{torsion pair in $\C$} if 
\be \item $\mathcal{C}(\mathcal{U}, \mathcal{V}) = 0$;
\item For all $C \in \mathcal{C}$ there exists an extriangle 
\begin{equation} \label{eqn:torsion extriangle}
\begin{tikzcd} U \arrow[r, "u"] & C \arrow[r, "v"] & V \arrow[r, dashed] &{} \end{tikzcd}
\end{equation}   with $U \in \U$ and $V \in \V$. 
\ee
If $\C$ is a triangulated category (viewed naturally as an extriangulated category) with shift functor $\Sigma$, a  torsion pair in $\C$, $(\mathcal{U}, \mathcal{V})$, is a \textit{t-structure}  \cite{ Beuilinson1982}  (resp. \textit{co-t-structure} \cite{Bondarko2010, Pauksztello2008}) if $\Sigma \mathcal{U} \subseteq \mathcal{U}$ (resp. $\Sigma^{-1} \mathcal{U} \subseteq \mathcal{U}$) with \textit{heart} $\mathcal{H} = \mathcal{U} \cap \Sigma \mathcal{V}$ (resp. \textit{co-heart} $\mathcal{M} = \mathcal{U} \cap \Sigma^{-1} \mathcal{V}$). We call the subcategory $\mathcal{U}$ the \textit{aisle} of the (co-)t-structure and $\mathcal{V}$ the \textit{co-aisle}. A (co-)t-structure is \textit{bounded} if the equalities $\C = \bigcup_{n \in \mathbb{Z}} \Sigma^n \U = \bigcup_{m \in \mathbb{Z}} \Sigma^m \V $ hold. 
\end{defn}

The reader should be aware that the terminology and notation of torsion pairs and (co-)t-structures in a triangulated category varies, often by a shift, throughout the literature and that co-t-structures were introduced under the name `weight structures' in \cite{Bondarko2010}.

\begin{rem} \label{rem:torsion properties} Let $(\U, \V)$  be a torsion pair in $\C = (\C, \EE, \fs)$. Then it follows quickly from the definition that the following properties hold. 
\be \item $\U$, $\V$ are closed under extensions.
\item $\U = \{ A \in \C \mid \C(A, \V) =0\}$ and $\V = \{A \in \C \mid \C(\U, A ) =0\}$. 
\item The morphism $u$ (resp. $v$) in the extriangle \ref{eqn:torsion extriangle} is a right $\U$ (resp. left $\V$) approximation of $C$. 
\ee
\end{rem}

Before we proceed, let us compare the above definition of a torsion pair in $\R$ with other notions appearing in the literature: \be 
\item A torsion pair $(\U, \V)$ in $\R$ is a \textit{right torsion pair}  \cite{Lin2007} if $\Sigma$ preserves $\V$-monics.
\item When $\R$ is equipped with a negative first extension structure, $\EE^{-1}$, (see Remark \ref{rem:negative extns}) a torsion pair $(\U, \V)$ in $\R$ is an \textit{$\fs$-torsion pair} \cite{Adachi2021} if $\EE^{-1}(\U, \V) =0$. In this case the extriangle \ref{eqn:torsion extriangle} is essentially unique and assigments $C \mapsto U$ and $C \mapsto V$ are functorial \cite[Proposition 3.7]{Adachi2021}.  
\ee
The next lemma shows that right torsion pairs and $\fs$-torsion pairs in $\R$ coincide (when we equip $\R$ with the natural first negative extension structure $\EE^{-1}(-.?) = \SR( -, \Sigma^{-1} ?)$) and that such torsion pairs remind us of t-structures; which are precisely the $\fs$-torsion pairs in a triangulated category. For a class of objects $\mathcal{X}$ in $\R$, by $\Sigma^{-1}\mathcal{X}$ we denote the class of objects $\{ X \in \R \mid \Sigma X \in \mathcal{X}\}$. 

\begin{lem}
Let $(\U, \V)$ be a torsion pair in $\R$. Then the following are equivalent
\be 
\item  $\Sigma \U \subseteq \U$;
\item  $\Sigma^{-1}\V \subseteq \V$;

\item $\R( \Sigma \U , \V)=0$; 
\item  $\Sigma$ preserves $\V$-monics. 
\ee
\end{lem}
\begin{proof}
\textbf{(a) $\Rightarrow$ (b):} 
Suppose that $\Sigma \U \subseteq \U$ and let $Y \in \Sigma^{-1}\V$. Then there is a  right triangle
\[ \begin{tikzcd} U \arrow[r, "u"] & Y \arrow[r] & V \arrow[r] &{} \end{tikzcd} \] with $U \in \U$ and $V \in \V$. Consider the morphism $\Sigma u : \Sigma U \to \Sigma Y$. By assumption, $\Sigma U \in \U$ and $\Sigma Y \in \V$; thus, $\Sigma u =0$. Since $\Sigma$ is faithful, $u =0$ and we deduce that $Y \cong V \in \V$. 

\textbf{(b)$\Rightarrow$ (c):} Follows from the fact that $\R(\U, \V)=0$ and $\Sigma$ is fully faithful. 

\textbf{(c)$\Rightarrow$(d):} Let $f:A \to B$ be a $\V$-monic morphism. There are right triangles 
\begin{align*}
&\begin{tikzcd}[ampersand replacement =\&] U \arrow[r, "u"] \& A \arrow[r, "v"] \& V \arrow[r] \&{\Sigma U} \end{tikzcd} \\
 \intertext{and} 
 &\begin{tikzcd}[ampersand replacement =\&] U' \arrow[r, "u'"] \& \Sigma A \arrow[r, "v'"] \& V' \arrow[r] \&{\Sigma U'} \end{tikzcd}
\end{align*} 
 with $U, U' \in \U$ and $V, V' \in V$. Then, as $f$ is $\V$-monic, there exists $g:B \to V$ such that $gf=v$. Since $\R(\Sigma \U, \V) =0$ there exists $a: \Sigma U \to U'$ such that $\Sigma u = u' a$. Thus, by the axiom (R3) = (ET3) there exists $c: \Sigma V \to V'$ such that $c (\Sigma v )= v'$ 
\[ \begin{tikzcd}\Sigma  U \arrow[r, "\Sigma u"] \arrow[d, "\exists a"]& \Sigma A \arrow[d, equal] \arrow[r, "\Sigma v"] & \Sigma V \arrow[d, "\exists c"] \arrow[r] &{{\Sigma^2 U}} \\ U' \arrow[r, "u'"] & \Sigma A \arrow[r, "v'"] & V' \arrow[r] &{{\Sigma U'}.} \end{tikzcd} \]
Thus 
\[ v' = c (\Sigma v) = c \Sigma (gf) = c (\Sigma g)(\Sigma f ).\] This finishes the proof since $v'$ is a left $\V$-approximation of $\Sigma A$ and so all morphisms from $\Sigma A$ to $\V$ factor through $v'$. 

\textbf{(d)$\Rightarrow$(a).} Let $U \in \U$. Then, since $\R(U, \V)=0$ we have that $U \to 0$ is $\V$-monic.  Thus, by assumption $\Sigma U \to 0$ is also $\V$-monic from which we deduce that $\Sigma U \in \U$. 

\end{proof}
 
 The next lemma justifies why we will look to describe $\R$ as a (co-)aisle in $\SR$.

\begin{lem} \label{lem:aisleinstab}
Let $\mathcal{T}$ be a triangulated category, $(\mathcal{V}, \mathcal{W})$ (resp.  $(\mathcal{U}, \mathcal{V})$) be a t-structure (resp. co-t-structure) in $\mathcal{T}$ and $\mathcal{S}:= \mathsf{cosusp}_\mathcal{T}\mathcal{V} \cong \mathcal{S}(\V)$. Then $(\mathcal{V}, \mathcal{W} \cap \mathcal{S})$ (resp. $(\mathcal{U} \cap \mathcal{S}, \mathcal{V})$) is a t-structure (resp. co-t-structure) in $\mathcal{T}$. 
\end{lem} 
\begin{proof}
We prove the t-structure case, whence statement for co-t-structures will follow dually. Let $(\mathcal{V}, \mathcal{W})$ be a t-structure in $\T$ and set $\mathcal{S}:= \mathsf{cosusp}_\mathcal{T}\mathcal{V} \cong \mathcal{S}(\V)$. 
Clearly, $\mathcal{S}(\mathcal{V}, \mathcal{W} \cap \mathcal{S})=0$ and $\Sigma \mathcal{V} \subset \mathcal{V}$. Thus it remains to show that $\mathcal{S} = \mathcal{V} \ast (\mathcal{W} \cap \mathcal{S})$. Since $\mathcal{V}, (\mathcal{W} \cap \mathcal{S}) \subset \mathcal{S}$ and $\mathcal{S}$ is closed under extensions, $ \mathcal{V} \ast (\mathcal{W} \cap \mathcal{S}) \subset \mathcal{S}$.
 To show the converse, let $A \in \mathcal{S}$ then there exists a triangle in $\mathcal{T}$ \[ \begin{tikzcd} V \arrow[r] & A \arrow[r] & W \arrow[r] & \Sigma V \end{tikzcd} \] with $V \in \mathcal{V}$ and $W \in \mathcal{W}$. Thus $W \in \mathcal{S} \ast \Sigma \mathcal{V} \subset \mathcal{S} \ast \mathcal{S} \subset \mathcal{S}$. 
\end{proof}

We also note that the boundedness of right triangulated categories (Lemma \ref{lem:bounded}) relates to the boundedness of (co-)t-structures.

\begin{lem} Suppose that $\R$ is the co-aisle of a co-t-structure, $(\U, \R)$ in $\SR$. If $\R$ is bounded then the co-t-structure $(\U, \R)$ is bounded. A dual statement holds for a t-structure $(\R, \V)$ in $\SR$. 
\end{lem}
\begin{proof} Let $(\U, \R)$ be a co-t-structure in  $\SR$.
By construction of $\S(\R)$, the equality $\S(\R) = \bigcup_{n \in \mathbb{Z}} \Sigma^n \R$ holds. It remains to verify that $\S (\R) = \bigcup_{m \in \mathbb{Z}} \Sigma^m \U $. Observe that since $\R$ is bounded, in light of Lemma \ref{lem:bounded}, for all $X \in \SR$ we have that $\SR(\R, \Sigma^i X) =0$ for $i<<0$. Thus, by Remark~\ref{rem:torsion properties}, $\Sigma^iX \in \V$ for $i<<0$; whence the claim follows. 
\end{proof}

Before we give our characterisations of $\R$ as the co-aisle of a co-t-structure, we make some comments. 
 Note, by Proposition \ref{funcyisom}, the $\EE$-projectives of $\R$ are the  precisely the objects $P \in \R$ satisfying $\R (P, \Sigma -) = 0$.  Recall from \cite{Aihara2012} that a subcategory $\mathcal{X} = \mathsf{add}\mathcal{X}$ of a triangulated category $\mathcal{T}$ is \textit{silting} if $\mathcal{T}(\mathcal{X}, \Sigma^{>0} \mathcal{X}) = 0$ and $\mathcal{T} = \mathsf{thick}\mathcal{X}$, the smallest triangulated subcategory of $\mathcal{T}$ containing $\mathcal{X}$ that is closed under direct summands. 
We are now ready to give characterisations of $\R$ as the co-aisle of a co-t-structure in $\SR$ in terms of torsion pairs, $\EE$-projectives of $\R$ and silting subcategories. 

\begin{thm} \label{thm:coaisles}
The following are equivalent
\be 
\item  $\R$ is the co-aisle of a co-t-structure $(\U, \R)$  in $\S(\R)$;
\item $\R$ has enough $\EE$-projectives;
\item There is a torsion pair $(\mathsf{Proj}_{\EE} \R, \Sigma \R)$ in $\R$.
\ee
Moreover, if $\R$ is bounded then the above conditions are also equivalent to \be \item[(d)] $\mathsf{Proj}_{\EE}\R$ is a silting subcategory of $\S(\R)$. \ee
\end{thm}
\begin{proof}
\textbf{(a) $\Rightarrow$ (b):} Suppose there is a co-t-structure $(\U, \R)$ in $\SR$. 
We must show that for all $X \in \R$ there is a right triangle 
\[ R \to P \to X \to \Sigma R  \] in $\R$ with $P \in \mathsf{Proj}\R$. Since $(\U, \R)$ is a co-t-structure in $\SR$ there is a triangle in $\SR$
\[ U \to \Sigma^{-1} X \to R \to \Sigma \R\] with $U \in \U$, $R \in \R$. By rotating we obtain the triangle
\[ R \to \Sigma U \to X \to \Sigma R.\] We claim that  $\Sigma U \in \mathsf{Proj}_{\EE}\R$. The fact that  $\Sigma U \in \R$ follows from this triangle 
 since $\R$ is closed under extensions in $\SR$ by Lemma \ref{lem: T extn closed}.
It remains to verify  that $\Sigma U$ is $\EE$-projective in $\R$, that is $\R(\Sigma U, \Sigma \R)=0$:
\[ 0 = \SR( U, \R) \cong \SR (\Sigma U , \Sigma \R) \cong \R( \Sigma U, \Sigma \R)\] where we have the first equality from the properties of co-t-structures.

\textbf{(b) $\Rightarrow$ (c):} Let $X \in \R$. By assumption there is a right triangle $R \to P \to X \to \Sigma \R$ in $\R$ with $P \in \mathsf{Proj}_{\EE}\R$. We rotate this to the right triangle \[ P \to X \to \Sigma R \to \Sigma P\] whence the claim follows as $\R(P, \Sigma \R)=0$.

\textbf{(c) $\Rightarrow$ (a):} We use a similar argument to \cite[Theorem 3.11]{MendozaHernandez2013}. Suppose that there is a torsion pair $(\mathsf{Proj}_{\EE}\R, \Sigma \R)$ in $\R$ and set $\PP := \mathsf{Proj}_{\EE}\R$ and $\U := \{Y \in \SR \mid \SR(Y, \R) =0\}$. Let $0\neq X \in \SR$, we must show that there is a triangle 
\[ U \to X \to R \to \Sigma U\] in $\SR$ with $U \in \U$ and $R \in \R$. By the definition of $\SR$ there exists $A \in \R \setminus \Sigma \R$ and $n \in \mathbb{Z}$ such that $X \cong \Sigma^n A $. If $n \geq 0$ then $X \in \R$ and we take the triangle $0 \to X \to X \to 0$. For $n < 0$ we proceed by induction.  By assumption, there is a right triangle in $\R$ 
\[ P \to A \to \Sigma R \to \Sigma P \] with $P \in \PP$ and $R \in \R$. This rotates to a triangle 
\[ \Sigma^n P \to \Sigma^n A \to \Sigma^{n+1} R \to \Sigma^{n+1} P\] in $\SR$. By the induction hypothesis, there is a triangle in $\SR$ 
\[ U \to \Sigma^{n+1} R \to V \to \Sigma U\] with $U \in \U$ and $V \in \R$. We apply (ET4)$^{\text{op}}$ to these triangles 
\begin{equation} \label{diagram:cot proof}
\begin{tikzcd}
\Sigma^n P \arrow[r] \arrow[d, equal] & E \arrow[r] \arrow[d]                      & U \arrow[d] \arrow[r]              & \Sigma^{n+1} P \arrow[d, equal] \\
\Sigma^n P \arrow[r]                                & \Sigma^{n} A \arrow[d] \arrow[r]           & \Sigma^{n+1} R \arrow[d] \arrow[r] & \Sigma^{n+1} P           \\
                                                    & V \arrow[r, equal] \arrow[d] & V \arrow[d]                        &                          \\
                                                    & \Sigma E \arrow[r]                         & \Sigma U                           &                         
\end{tikzcd}
\end{equation}
and claim that the triangle $E \to \Sigma^n A \to V \to \Sigma E$ satisfies the required conditions. Since $V \in \R$ by construction, we only have to show that $E \in \U$. Observe that $\Sigma^n P$ is in $\U$: 
\[ \SR( \Sigma^n P, \R) \cong \SR( P, \Sigma^{-n} \R) \cong \R ( P, \Sigma^{-n} \R ) =0\] since $\R(P, \Sigma \R) =0$ and $n < 0$. Thus the top row of the Diagram \ref{diagram:cot proof} shows that $E \in \U$ since $\U$ is closed under extensions. 

\textbf{(d)$\Leftrightarrow$(a):} Since $\R$ is bounded, by Lemma \ref{lem:bounded} the co-t-structure $(\U, \V)$ is bounded in $\SR$. Observe that $\mathsf{Proj}_{\EE}\R$ is the co-heart of this co-t-structure. The claim then follows from \cite[Corollary 5.9]{MendozaHernandez2013} where it was shown that a subcategory of a triangulated category is silting precisely when it is the co-heart of a bounded co-t-structure. 
\end{proof}

As a consequence, we obtain the following. 

\begin{cor} \label{cor:siltingcorrespondence} There is a correspondence between silting subcategories of triangulated categories and bounded right triangulated categories with right semi-equivalence that have enough projectives. \end{cor}

The characterisations of Theorem \ref{thm:coaisles} adds to the numerous interpretations of silting subcategories in triangulated categories which are surveyed in \cite{AngeleriHuegel2019a}.

We now present our characterisations of aisles of t-structures. 

\begin{thm} \label{thm:aisles}
The following are equivalent
\be 
\item $\R$ is the aisle of a t-structure $(\R, \V)$  in $\S(\R)$;
\item There is a torsion pair $(\Sigma \R, \mathcal{F})$ in $\R$;
\item There is an equivalence of triangulated categories $\phi: \SR \to \C(\R)$  
\[\begin{tikzcd}[column sep=small] 
 & \R \arrow[dl, "s"'] &  \\
\SR \arrow[rr, "\phi", "\cong"']  & & \C(\R) \arrow[ul, "c"'] \end{tikzcd} \] such that $c\phi s = 1_\R$. 
\ee
\end{thm}
\begin{proof}
The equivalence of (a) and (b) follows from dualising the arguments used in the proof of Theorem~\ref{thm:coaisles}. 

\textbf{(a)$\Rightarrow$(c):} We show that in this case, $\SR$ satisfies the universal property of the co-stabilsation of $\R$. Let $r: \SR \to \R$ denote the functor induced by the t-structure $(\R, \V)$. Let $\T$ be a triangulated category and $F: \T \to \R$ be a right triangle functor. Define the functor $F': \T \to \SR$ by $F'X = (FX, 0)$. Clearly $rF' = F$ and $F'$ is a triangle functor. 

We quickly verify that $F'$ is unique with this property. Suppose that $F'': \T \to \R$ is another triangle functor such that $rF''=F$. Then for all $X \in \T$ there are triangles in $\SR$
\begin{align*}
rF'X \cong FX &\to F'X \to V' \to \Sigma FX\\
rF''X \cong FX &\to F''X \to V'' \to \Sigma FX
\end{align*}
 with $\SR(\R, V') = 0 = \SR (\R, V'')$. Observe that,  for $n>>0$, $\Sigma^n F'X, \Sigma^n F''X \in \R$. Thus, $\Sigma^n F'X \cong \Sigma^n FX \cong \Sigma^n F''X$.
 So, since $F'$ and $F''$ are triangle functors, $F''\Sigma^n X \cong F' \Sigma^n X$ and we are done. 

\textbf{(c)$\Rightarrow$(a):}  We will show that that $c\phi$ is right adjoint to the inclusion $s$ whence we will be done by \cite[Proposition 1.2]{Keller1988}. It is enough to show that for all $X, Y \in \R$ and $m \in \mathbb{Z}$ that there is a natural isomorphism 
\[ \R(X, c\phi (Y,m) ) \cong \SR(sX, (Y,m)). \]
If $m \geq 0$ then $(Y,m) = s(\Sigma^m Y)$ and thus, by assumption $c\phi(Y,m) = c\phi s (\Sigma^m Y) \cong \Sigma^m Y$. On the other hand
$$\SR(sX, (Y,m)) = \SR( (X,0), (Y,m)) \cong \R(X, \Sigma^m Y)$$ and we are done. 

If $m <0$, then $\SR( sX, (Y,m) ) \cong \R(\Sigma^{-m}X, Y)$. On the other hand, 
$
 \R(X, c \phi (Y,m))  \cong \R( \Sigma^{-m}X, \Sigma^{-m}c\phi (Y,m)) 
$ and $$ \Sigma^{-m}c\phi (Y,m) \cong c\phi \Sigma^{-m}(Y,m) = c \phi (Y,0) \cong c \phi s Y \cong Y$$ and we are done since all isomorphisms used are natural. 
\end{proof}

We end this section by noting that the language of the torsion pairs in $\R$ also allows us to describe related  classes of (co-)t-structures.

\begin{prop} \label{prop:intermediate} $\R = (\R, \Sigma, \Delta)$  be a right triangulated category with right semi-equivalence.
\be 
\item Suppose that there is a co-t-structure $(\U, \R)$ in $\SR$. Then there is a bijection 
\begin{align*}
\{ (\mathcal{X}, \mathcal{Y}) \text{ co-t-structure in } \SR  & \mid \mathcal{Y} \subseteq \R \}    \\ \longleftrightarrow & \{ (\mathcal{A}, \mathcal{B}) \text{ torsion pair in } \R \mid \Sigma \mathcal{B} \subseteq \mathcal{B} \subseteq \Sigma \R \} 
\\ (\X , \mathcal{Y})  \longmapsto & (\R \cap \Sigma \X, \Sigma \mathcal{Y}) 
\\ (\Sigma^{-1}(\U \ast \mathcal{A}), \Sigma^{-1} \mathcal{B}) \longmapsfrom & (\mathcal{A}, \mathcal{B}) 
\intertext{ which preserves the inclusion of aisles and restricts to a bijection }
\{ (\mathcal{X}, \mathcal{Y}) \text{ co-t-structure in } \SR   & \mid \Sigma \R \subseteq \mathcal{Y} \subseteq \R \}   \\ \longleftrightarrow &  \{ (\mathcal{A}, \mathcal{B}) \text{ torsion pair in } \R \mid \Sigma^2 \mathcal{R} \subseteq \mathcal{B} \subseteq \Sigma \R \}.
\end{align*}
\item  Suppose that there is a t-structure $(\R, \V)$ in $\SR$. Then there is a bijection 
\begin{align*}
\{ (\mathcal{X}, \mathcal{Y}) \text{ t-structure in } \SR  & \mid \mathcal{X} \subseteq \R \}    \\ \longleftrightarrow & \{ (\mathcal{A}, \mathcal{B}) \text{ torsion pair in } \R \mid \Sigma \mathcal{A} \subseteq \mathcal{A} \subseteq \Sigma \R \} 
\\ (\X , \mathcal{Y})  \longmapsto & (\Sigma \X, \R \cap \Sigma \mathcal{Y}) 
\\ (\Sigma^{-1} \mathcal{A}, \Sigma^{-1}(\V \ast \mathcal{B}) )\longmapsfrom & (\mathcal{A}, \mathcal{B}) 
\intertext{ which preserves the inclusion of aisles and restricts to a bijection }
\{ (\mathcal{X}, \mathcal{Y}) \text{ t-structure in } \SR   & \mid \Sigma \R \subseteq \mathcal{X} \subseteq \R \}   \\ \longleftrightarrow &  \{ (\mathcal{A}, \mathcal{B}) \text{ torsion pair in } \R \mid \Sigma^2 \mathcal{R} \subseteq \mathcal{A} \subseteq \Sigma \R \}. \end{align*}
\ee
\end{prop}
\begin{proof} We do not write the proof of (a) since it is a straightforward generalisation of the proof of \cite[Theorem 4.4]{tattar2020torsion}. Part (b) can be proved in a similar way. Alternatively, since all torsion pairs in (b) are $\fs$-torsion pairs, the bijections of (b) are, after shifting, special cases of \cite[Theorem 3.9]{Adachi2021}.
\end{proof}
The (co-)t-structures in the second correspondences in parts (a) and (b) are called \textit{intermediate} with respect to the (co-)t-structure of $\R$.  Let us mention related work on this property. 
Intermediate t-structures also correspond to torsion pairs in the heart \cite{Happel1996, beligiannis2007homological} and have applications in, for example, stability conditions \cite{Woolf2010} and algebraic geometry \cite{Polishchuk2007}.
In the case of co-t-structures, it has been shown that intermediate co-t-structures correspond to certain two term silting subcategories \cite{Iyama2014} and to cotorsion pairs in the `extended coheart' 
\cite{pauksztello2020co}.

\section{Aisles through quotients} \label{sec:aislesquotients}

We fix some conventions. Let $\C = (\C, \EE, \fs)$ be a Frobenius extriangulated category, with projectives-injectives $\I$ and $\underline{\mathcal{C}} = \underline{\mathcal{C}}_{\I}$ be the stable category, which has the natural structure of triangulated category \cite[Corollary 7.4]{nakaoka2019extriangulated} with shift functor $\Sigma$. By $\I(X)$ we denote the object such that $X \to \I(X)$ is a minimal left $\I$-approximation. 

In this section, we show the following.  

\begin{thm} \label{thm:quotientshiftaisle}
Let $(\U, \V)$ be a t-structure in $\underline{\C}$ and set $\D = \Sigma^{-1} \V$. Then there is an equivalence of right triangulated  categories $\Sigma_\D \underline{\C}_\D \cong \U$. 
\end{thm} 

\begin{rem} \label{rem:SS,Ncomplementary}
In \cite[Proposition 3.9]{saorin2011exact}, where it was shown that t-structures in an algebraic triangulated category correspond bijectively to certain complete cotorsion pairs in the associated Frobenius exact category. 
Along that bijection a t-structure $(\U, \V)$ in $\und{\C}$ corresponds to the cotorsion pair $( \U, \Sigma^{-1} \V)$ in $\C$. This bijection generalises immediately to the extriangulated setting.  Further, when then applied to the case of a triangulated category $\mathcal{T}$ (which is Frobenius extriangulated category \cite[Proposition 3.22]{nakaoka2019extriangulated}), this specialises to the observation of \cite[Proposition 2.6]{Nakaoka2011} that t-structures in $\T$ are in bijection with cotorsion pairs $(\U, \V)$ in $\T$ such that $\Sigma \U \subseteq \U$. Thus, these results and Theorem \ref{thm:quotientshiftaisle} complement each other.
 \end{rem}

 We note that a class of objects $\V \subset \mathsf{Obj}(\C) = \mathsf{Obj}(\underline{\C})$ defines subcategories of both $\C$ and $\underline{\mathcal{C}}$. 

\begin{lem} \label{liftingcov}
Let $\D = \mathsf{Add}\D$ be a covariantly finite additive subcategory of $\underline{\mathcal{C}}$. Then $\D$ is also a covariantly finite additive subcategory of $\mathcal{C}$. Moreover, left $\D$-approximations are $\EE$-inflations in $\C$.
\end{lem}
\begin{proof}
For any $X \in \mathcal{C}$, let $\underline{\alpha}: X \to D$ and $\beta: X \to \I(X) = I$ be a left $\D$-approximation of $X$ in $\underline{\mathcal{C}}$ and a left $\I$-approximation of $X$ in $\mathcal{C}$ respectively. We claim that $\left[ \begin{smallmatrix} \alpha \\ \beta \end{smallmatrix} \right] : X \to D \oplus I$ is a left $\D$-approximation of $X$ in $\mathcal{C}$.

Indeed, since $\D$ is additive, $0 \cong \I\subset \D$, so $D \oplus I \in \D$. Let $f: X \to W$ be a morphism in $\mathcal{C}$ with $W \in \D$. Then there exists $(\und{g}:D \to W) \in \und{\C}(D,W)$ such that $\und{f}=\und{g \alpha}$. Now, $\und{f - g \alpha} =0$ so $f - g \alpha$ factors through $\I$ and so must factor through the left $\PP$-approximation of $X$. Thus there exists $g':I \to W$ such that $f - g \alpha = g' \beta$. Rearranging we have
\[ \begin{tikzcd}[ampersand replacement=\&] 
X \arrow[r, "{ \left[ \begin{smallmatrix} \alpha \\ \beta \end{smallmatrix} \right] }"] \arrow[d, "f"'] \& D \oplus I \arrow[dl,  "{ \left[ \begin{smallmatrix} g & g' \end{smallmatrix} \right] }"] \\ W. \& \end{tikzcd} \] The remaining claim follows from Lemma \ref{lem:LN1.20}.
\end{proof}

Let $(\U, \V)$ be a t-structure on $\und{\C}$, so that for all objects $X$ there exists a unique (up to isomorphism) triangle in $\und{\C}$
\[
\begin{tikzcd} \U(X) \arrow[r, "\und{u}_X"] & X \arrow[r, "\und{v}_X"] & \V(X) \arrow[r, "\und{w}_X"] & \Sigma \U(X) \end{tikzcd}
\]  therefore there is a unique (up to isomorphism) triangle in $\und{\C}$
\begin{equation}\label{tstrtriangle}
\begin{tikzcd} \Sigma^{-1} \U(\Sigma X) \arrow[r, "\und{u}'_X"] & X \arrow[r, "\und{v}'_X"] & \Sigma^{-1} \V(\Sigma X) \arrow[r, "\und{w}'_X"] & \U(\Sigma X). \end{tikzcd}
\end{equation}
Note that this triangle is also the canonical triangle of $X$ with respect to the t-structure $(\Sigma^{-1} U, \Sigma^{-1} \V)$. Set $\D := \Sigma ^{-1}\V \subset \C$, which by Lemma \ref{liftingcov} satisfies property $(\ast )$ and thus the category $\und{\C}_\D$ is a right triangulated category by Proposition~\ref{thm:abm}. We also set $\V'(X) :=  \Sigma^{-1} \V(\Sigma X)$ and $J(X):=  \V' (X)\oplus \I(X)$.

\begin{lem} \label{lem:shiftgood}
For all $X \in \mathcal{C}$, $\Sigma_\mathcal{D} X \cong \U(\Sigma X)$ in $\und{\C}$. 
\end{lem}
\begin{proof} 
By the definition of the functor $\Sigma_\D$ and  Lemmas \ref{lem:LN1.20} and \ref{liftingcov}, $\Sigma_\D X$ fits into a commutative diagram of extriangles in $\C$
\[ \begin{tikzcd}  
X \arrow[r] \arrow[d, "v'_X"']   & \I(X) \arrow[r] \arrow[d] & \Sigma X \arrow[d, equal] \arrow[r, dashed] & {} \\ \V' (X) \arrow[r] & \Sigma_\D X \arrow[r] & \Sigma X \arrow[r, dashed] & {} .
\end{tikzcd} \] where $\und{v}'_X:X \to \V'(X)$ fits into the triangle (\ref{tstrtriangle}). Thus $\Sigma_\D X = \mathsf{cone}_{\und{\C}}(\und{v}'_X) \cong  \U (\Sigma X)$ in $\und{\C}$. 
\end{proof}

Consider the diagram 
\[\begin{tikzcd}
 & \C \arrow[r, "\mathsf{can}"] \arrow[d, "\mathsf{can}"'] & \und{\C}_\D \\
\U \arrow[r, "i", "\mathsf{can}"', hook] & \und{\C} \arrow[ru, "\exists G"']                       &           
\end{tikzcd} \]
where $G$ is the unique additive functor rendering the diagram commutative which exists since $\I \subseteq \D$. The next result finishes the proof of Theorem \ref{thm:quotientshiftaisle}.

\begin{prop} \label{thm: alg aisle} The composition $\psi := Gi: \U \to  \und{\C}_\D$ is a full and faithful right triangle functor and induces an equivalence of right triangulated catgeories $U \cong \Sigma_\D \und{\C}_\D$. \end{prop}
\begin{proof}
\textbf{$\psi$ is a right triangulated functor:}
Note that for all  $U \in \U$, $\U(U) \cong U$ and $\V'(U) \cong 0$ in $\und{\C}$. Thus, by Lemma \ref{lem:shiftgood}, there is a natural isomorphism of functors $\Sigma |_{\U} \cong \Sigma_\D$ and $I(U) \cong J(U)$ so that (right) triangles in $\U$ are also right triangles in $\und{C}_\D$.

\textbf{$\psi$ is full:}
Follows from the fact that $i$ and $G$ are both full.

\textbf{$\psi$ is faithful:}
Let $\und{f} :X \to Y$ be a morphism in $\U$. Suppose that $\psi (\und{f}) = 0$. Then $\und{f}$ must factor through $\D = \Sigma^{-1} \V$ in $\und{C}$. But since $\und{C}(\U, \V)=0$ and $\Sigma ^{-1} \V \subset \V$ this is only possible if $\und{f}=0$.

\textbf{$\mathsf{Im}\psi \cong \Sigma_\D \und{\C}_\D$:} Let $\Sigma_\D X \in \Sigma_\D \und{C}_\D$, by Lemma \ref{lem:shiftgood}, $\Sigma_\D X \cong \U (\Sigma X)$ in $\C$ and hence also in $\und{C}_\D$. Thus $\psi (\U (\Sigma X)) \cong \Sigma_\D X$. Conversely, let $Y \in \U$. Then $Y \cong \Sigma_\D (\Sigma^{-1} Y) \in \und{\C}_\D$ by Lemma \ref{lem:shiftgood}.
\end{proof}

\begin{cor}
$\Sigma_\D G \Sigma^{-1}: \und{C} \to \Sigma_\D \und{C}_\D \cong \U$ is right adjoint to the inclusion $\U \hookrightarrow \und{C}$. In particular, $\Sigma_\D G \Sigma^{-1} \cong \U(-)$.
\end{cor}

\begin{rem}
Theorem \ref{thm:quotientshiftaisle} cannot be directly dualised to give co-aisles of co-t-structures since the aisle of a co-t-structure is contravariantly finite. The author would be interested to know if there is a way to work around this problem.  
\end{rem}

\bibliographystyle{plain}
\bibliography{rib}

 \medskip
  \begin{tabular}{@{}l@{}}%
  \textsc{Department of Mathematics,  Universit\"{a}t zu K\"{o}ln,} \\ 
\textsc{Weyertal 86-90, 50931 K\"{o}ln, Germany}
\\
    \texttt{atattar@uni-koeln.de}
  \end{tabular}

\end{document}